\documentclass[reqno]{amsart}
\usepackage{amsmath}
\usepackage{amssymb, hyperref}
\usepackage{mathtools}
\mathtoolsset{showonlyrefs}

\usepackage{graphicx,color}
\renewcommand{\Re}{\operatorname{Re}}
\renewcommand{\Im}{\operatorname{Im}}
\newcommand{\Real}{\mathbb R}
\begin{document} 
\title[Inhomogeneous NLS with inverse-square potential]{On the inhomogeneous NLS with inverse-square potential }
	\author[L. CAMPOS ]
	{LUCCAS CAMPOS}  
	
	\address{Luccas Campos  \hfill\break
	Department of Mathematics, UFMG, Brazil.}
	\email{luccasccampos@gmail.com}
	
	\author[C. M. GUZM\'AN ]
	{CARLOS M. GUZM\'AN } 
	
	\address{CARLOS M. GUZM\'AN \hfill\break
		Department of Mathematics, Fluminense Federal University, BRAZIL}
	\email{carlos.guz.j@gmail.com}

\begin{abstract}
We consider the inhomogeneous nonlinear Schr\"odinger equation with inverse-square potential in $\mathbb{R}^N$
$$
i u_t + \mathcal{L}_a u+\lambda |x|^{-b}|u|^\alpha u = 0,\;\;\mathcal{L}_a=\Delta -\frac{a}{|x|^2},
$$
where $\lambda=\pm1$, $\alpha,b>0$ and $a>-\frac{(N-2)^2}{4}$. We first establish sufficient conditions for global existence and blow-up in $H^1_a(\Real^N)$ for $\lambda=1$, using a
Gagliardo-Nirenberg-type estimate. In the sequel, we study local and global well-posedness in $H^1_a(\Real^N)$ in the $H^1$-subcritical case, applying the standard Strichartz estimates combined with the fixed point argument. The key to do that is to establish good estimates on the nonlinearity. Making use of these estimates, we also show a scattering criterion and construct a wave operator in $H^1_a(\Real^N)$, for the mass-supercritical and energy-subcritical case.  
\end{abstract}

\maketitle  
	\numberwithin{equation}{section}
	\newtheorem{theorem}{Theorem}[section]
	\newtheorem{proposition}[theorem]{Proposition}
	\newtheorem{lemma}[theorem]{Lemma}
	\newtheorem{corollary}[theorem]{Corollary}
	\newtheorem{remark}[theorem]{Remark}
	\newtheorem{definition}[theorem]{Definition}
	
\section{Introduction}
\indent In this paper, we study the initial value problem (IVP) associated to the inhomogeneous nonlinear Schr\"odinger equation (INLS$_a$ for short)
\begin{equation}\label{INLSa}
\begin{cases}
i\partial_tu-\mathcal{L}_a u + \lambda|x|^{-b}|u|^\alpha u =0,  \;\;\;t\in \mathbb{R} ,\;x\in \mathbb{R}^N,\\
u(0,x)=u_0(x), 
\end{cases}
\end{equation}
where $N\geq 3$, $\mathcal{L}_a=-\Delta +\frac{a}{|x|^2}$ with $a>-\frac{(N-2)^2}{4}$ and $\alpha, b>0$. The equation is called ``focusing INLS$_a$" when $\lambda= +1$ and ``defocusing INLS$_a$" when $\lambda= -1$. The restriction on $a$ comes from the sharp Hardy inequality, namely
\begin{equation}
\frac{(N-2)^2}{4}\int |x|^{-2}|u(x)|^2 dx \leq \int |\nabla u(x)|dx,\;\;\forall\;u\in H^1,  
\end{equation}
which guarantees that $\mathcal{L}_a$ is a positive operator.
It is a model from various physical contexts, for example, in nonlinear optical systems with spatially dependent interactions (see e.g. \cite{belmonte2007lie} and the references therein). In particular, when $a=0$, it can be thought of as modeling
inhomogeneities in the medium in which the wave propagates (see for instance \cite{MALOMED1}). When $b=0$, the equation \eqref{INLSa} appears in several physical settings, such a  quantum field equations or black hole solutions of the Einstein’s equations (see e.g. \cite{MR0397192}).   

The INLS$_a$ model can be seen as an extension to the Schr\"odinger equation. For instance, when $a=b=0$ we have the classical nonlinear Schr\"odinger (denoted by NLS) equation, extensively studied over the three decades. (see \cite{Bourgain}, \cite{CAZENAVEBOOK}, \cite{FELGUS}, \cite{Tao} and the references therein). The case $b=0$ and $a\neq 0$ is known as the NLS equation with inverse square potential, denoted by NLS$_{a}$ equation, that is,
\begin{equation}\label{NLSa}
i\partial_tu-\mathcal{L}_a u + |u|^\alpha u =0.
\end{equation}
This equation has also been intensively studied in recent years (see for instance, \cite{Burq}, \cite{Murphy}, \cite{Murphy1}, \cite{Okazawa2012}, \cite{Zhang}). Moreover, when $a=0$ and $b\neq 0$ we have the inhomogeneous NLS equation, denoted by (INLS), i.e.,
\begin{equation}\label{INLS}
i\partial_tu+\Delta u +|x|^{-b} |u|^\alpha u =0,
\end{equation}
which also has received substantial attention recently (see e.g. \cite{CARLOS}, \cite{paper2}, \cite{paper3}, \cite{Campos}, \cite{Dinh2017}, \cite{Murphy}).  

Similarly to the INLS, the INLS$_a$ equation is invariant under the scaling, $u_\mu(t,x)=\mu^{\frac{2-b}{\alpha}}u(\mu^2 t,\mu x)$ for $\mu >0$. 
A straightforward computation yields
$$
\|u_{0,\mu}\|_{\dot{H}^s}=\mu^{s-\frac{N}{2}+\frac{2-b}{\alpha}}\|u_0\|_{\dot{H}^s},
$$
implying that the scale-invariant Sobolev space is $\dot{H}^{s_c}(\mathbb{R}^N)$, with $s_c=\frac{N}{2}-\frac{2-b}{\alpha}$, the so called \textit{critical Sobolev index}. If $s_c = 0$ (or $\alpha = \frac{4-2b}{N}$) the (IVP) is known as mass-critical or $L^2$-critical; if $s_c<0$ (or $0<\alpha<\frac{4-2b}{N}$) it is called mass-subcritical or $L^2$-subcritical; if $0<s_c<2$, \eqref{INLSa} is known as mass-supercritical and energy-subcritical (or intercritical). In additional, 
solutions to \eqref{INLSa} conserve their mass and energy, defined respectively by 
\begin{equation}\label{mass}
M[u(t)]=\int_{\mathbb{R}^N}|u(t,x)|^2dx=M[u_0],
\end{equation}
\begin{equation}\label{energy}
E_a[u(t)]=\frac{1}{2}\int_{\mathbb{R}^N}| \nabla u(t,x)|^2dx+
\frac{a}{2}\int_{\mathbb{R}^N}|x|^{-2}| u(t,x)|^2dx+
\frac{\lambda}{\alpha +2} \int_{\mathbb{R}^N} |u|^{\alpha +2}dx=E_a[u_0].
\end{equation}
When $a=0$, we denote $E_a$ by $E_0$. 

\ We review now some recent developments in $H^1(\Real^N)$, for the particular cases, i.e., NLS$_{a}$ and INLS models, starting with the NLS$_a$ equation. Okazawa-Suzuki-Yokota \cite{Okazawa2012}, by the energy method, showed local well-posedness in the energy-subcritical case, for $N\geq 3$ and $a>-\tfrac{(N-2)^2}{4}$.
They also proved that the solutions are global if $\lambda=-1$ or $0<\alpha< \frac{4}{N}$ and $\lambda=1$. 
Zhang-Zheng \cite{Zhang} studied the defocusing case, establishing well posedness and scattering for $\tfrac{4}{N}<\alpha<\tfrac{4}{N-2}$, assuming $a \geq 0$ and $N = 3$, or $a\geq  -\tfrac{(N-2)^2}{4}+\tfrac{4}{(\alpha+2)^2}$ and $N\geq 4$. 
Recently, Killip-Murphy-Visan-Zheng in \cite{Murphy} considered the focusing $3D$ cubic NLS$_a$. They established well-posedness and scattering in $H^1(\mathbb{R}^N)$ if $a>-\tfrac{1}{4}$. Later, Lu-Miao-Murphy \cite{Murphy1} extended the result of \cite{Murphy} to all $L^2$-supercritical, energy subcritical nonlinearities, in dimensions, $3\leq N\leq 6$, with
\begin{equation}\label{conditionNLSa}
\begin{cases} a>-\tfrac{1}{4}\;\;\;\;\;\qquad \qquad\qquad \qquad \textnormal{if}\;\;\;N=3,\;\;\tfrac{4}{3}<\alpha\leq 2\\
a>-\frac{(N-2)^2}{4}+(\frac{N-2}{2}-\frac{1}{\alpha})^2\;\;\textnormal{if}\;\;\;3\leq N\leq 6,\;\; \max\{\frac{2}{N-2},\frac{4}{N}\}<\alpha< \frac{4}{N-2}.
\end{cases}
\end{equation}
On the other hand, the INLS equation was first studied by Genoud-Stuart \cite{GENSTU} via the energy method. For $N\geq 1$ and $0<b<\min\{N,2\}$, they showed local well posedness in $H^1(\Real^N)$ for the $H^1$-subcritical case and global well possedness in the mass-subcritical case. In the mass-critical case, Genoud in \cite{GENOUD} established global well-posedness in $H^1(\mathbb{R}^N)$, provided that the mass of the initial data is below that of the associated ground state. This result was extended in the intercritical case by Farah \cite{LG} and he also showed that the solution blows-up in finite time (see also \cite{Dinh-blowup}). The second author in \cite{CARLOS}, using Kato's method, proved local well-posedness in $H^1(\mathbb{R}^N)$, for the energy subcritical case in dimensions $N\geq 4$ and $0 <b < 2$. Cho-Lee \cite{Cho-Lee} treated the case $N=3$ for $0<b<\frac{3}{2}$ and Dinh \cite{Dinh2017} the case $N=2$ for $0<b<1$. Furthermore, in the intercritical case, the second author in \cite{CARLOS} also established a small data global theory in $H^1(\mathbb{R}^N)$ for $N\geq 4$, the first author \cite{Campos} treated the case $N=3$ and Cardoso-Farah-Guzmán \cite{Cardoso-Farah-Guzman} the case $N=2$. In all these works the range of $b$ is the same one where local well-posedness was obtained.

Motivated by the aforementioned papers, our main interest in this paper is to prove similar results for the INLS$_a$ equation in $H^1_a(\Real^N)$. The well-posedness theory was already studied by 
Suzuki \cite{Suzuki}. Using the energy method, the author showed that, if\footnote{It is worth mentioning that in \cite{Suzuki} the author considered \eqref{INLSa} with $a=-\frac{(N-2)^2}{4}$, the critical coefficient. The proof for the case $a>\frac{(N-2)^2}{4}$ is an immediate consequence of the previous one. } 
\begin{equation}\label{LWPEM}
N\geq 3,\qquad 0<\alpha < \tfrac{4-2b}{N-2}\qquad a> -\tfrac{(N-2)^2}{4}, \qquad \textnormal{and}\qquad 0<b<2, \end{equation}
then \eqref{INLSa} is locally well-posed in $H^1_a(\mathbb{R}^N)$ satisfying $
u \in  C \left([0, T);H^1(\mathbb{R}^N)\right) \cap C^1 \left([0, T);H^{-1}(\mathbb{R}^N)\right)$  for some $T > 0$.
It was also proved that any local solution of the IVP (\ref{INLSa}) with $u_0\in H^1_a(\mathbb{R}^N)$ extends globally in time if either $\lambda=-1$ (defocusing case) or $0<\alpha<\tfrac{4-2b}{N}$ for $\lambda=1$ (focusing, $L^2$-subcritical case). 

Our first goal here is to study global existence and blow-up in $H^1(\Real^N)$ for $\lambda=1$ for both the $L^2$-critical and the intercritical ($L^2$-supercriticial and $H^1$-subcritical) case. For that, we apply a sharp Gagliardo-Nirenberg type estimate. In order to do so, we prove the existence of a ground state.

\begin{proposition}\label{exist_ground} Let $N \geq 3$ and $\alpha$, $a$, $b$ as in \eqref{LWPEM} and $\lambda=1$. There exists a positive solution to the elliptic equation $$ \mathcal{L}_a Q+Q-|x|^{-b}|Q|^{\alpha}Q=0$$ in $H^1_a(\mathbb{R}^N)$. Moreover, all possible solutions have the same mass, the same $\dot{H}_a^1$ norm and the same energy.
\end{proposition}

Using the variational analysis associated to Proposition \ref{exist_ground}, we have thresholds for global existence and blow-up. 

\begin{theorem}[$L^2$-critical case]\label{GBCrit}
Let $N$, $a$, $b$ as in \eqref{LWPEM}, $\lambda=1$ and $\alpha = \frac{4-2b}{N}$. Suppose that $u$ is the solution to \eqref{INLSa} with initial data $u_0 \in H^1_a(\Real^N)$. If 

\begin{itemize}
    \item[a)]  (Global existence) $\|u_0\|_{L^2} < \|Q\|_{L^2}$, then the solution is uniformly bounded in $H^1_a$, and therefore extends globally in time;
    \item[b)] (Blow-up) $E[u_0]<0$ and either $|x|u_0 \in L^2(\Real^N)$ or $u_0$ is radial, then $u$ blows-up in finite positive and negative times.
\end{itemize}
\end{theorem}

\begin{theorem}[Intercritical case]\label{GBInter} Let $N$, $a$, $b$ as in \eqref{LWPEM}, $\lambda=1$ and $\frac{4-2b}{N}< \alpha <\frac{4-2b}{N-2}$. Suppose that $u$ is the solution of \eqref{INLSa} with initial data $u_0 \in H^1_a(\Real^N)$. Assume
\begin{equation}
    M(u_0)^{1-s_c} E_a(u_0)^{s_c} < M(Q)^{1-s_c}E_a(Q)^{s_c}.
\end{equation}
If
\begin{itemize}
    \item[a)] (Global existence)  
    \begin{equation}\label{MK<1}\|u_0 \|^{1-s_c}_{L^2}\| \sqrt{\mathcal{L}_a}u_0 \|^{s_c}_{L^2}  < \|Q\|^{1-s_c}_{L^2} \|\sqrt{\mathcal{L}_a} Q\|^{s_c}_{L^2},
\end{equation}
then the solution is uniformly bounded in $H^1_a$, and therefore extends globally in time;
    \item[b)] (Blow-up)
\begin{equation}
\|u_0 \|^{\frac{1-s_c}{s_c}}_{L^2}\|\sqrt{\mathcal{L}_a} u(t) \|_{L^2}  > \|Q\|^{1-s_c}_{L^2} \|\sqrt{\mathcal{L}_a} Q\|^{s_c}_{L^2}
\end{equation}
and, in addition, if $|x|u_0 \in L^2(\Real^N)$ or $u_0$ is radial, then $u$ blows-up in finite positive and negative times.
\end{itemize}

\end{theorem}


The main tools for proving Theorems \ref{GBCrit} and \ref{GBInter} are the coercivity given by the variational analysis,  and the virial identities. The main difficulty is to appropriately control the error terms appearing when one truncates the virial identity in the radial case. The $L^2$-critical case is especially delicate, since the criticality gives us less room for error terms.

Note that the local well-posedness showed in \cite{Suzuki} as well as the global results, Theorems \ref{GBCrit}a) and \ref{GBInter}a) ensure the existence of solutions to \eqref{INLSa}. However we do not know
whether or not the solutions satisfy $u \in L^q(I;H^{1,r}_a)$ for any $L
^2$-admissible pair ($q,r$), which is a key property to study other problems such as scattering for example. To obtain this extra information and working towards the proof of scattering in a next work, we establish the local and global well-posedness for \eqref{INLSa} via Kato’s method, which is based on the contraction mapping principle and the Strichartz estimates. We start with the local theory for the energy-subcritical case.
\begin{theorem}\label{LWP}
Assume that $N \geq 3 $ and $0\leq b<\min\{\frac{N}{2},2\}$. If $u_0 \in H^1_a(\mathbb{R}^N)$ and
\begin{equation}\label{conditionINLSa}
\begin{cases} a>-\tfrac{(N-2)^2}{4}\;\;\;\;\qquad \qquad \qquad \qquad\ \textnormal{if}\;\;\;\;\tfrac{2-2b}{N}<\alpha\leq\; \frac{2-2b}{N-2}\quad \textnormal{and} \;\; 0\leq b<1,\\
a>-\tfrac{(N-2)^2}{4}+\left(\tfrac{\alpha(N-2)-(2-2b)}{2(\alpha+1)}\right)^2\;\;\textnormal{if}\;\;\; \max\{0,\tfrac{2-2b}{N-2}\}<\alpha< \tfrac{4-2b}{N-2}.
\end{cases}
\end{equation}
Then there exists $T=T(\|u_0\|_{H^1_a},N,\alpha,b)$ and a unique solution of \eqref{INLSa} satisfying
$$
u \in C\left([0,T];H^1_a(\mathbb{R}^N) \right) \cap L^q\left([0,T];H^{1,r}_a(\mathbb{R}^N)    \right),
$$
where ($q,r$) is any $L^2$-admissible pair. 
\end{theorem}

New challenges and technical obstructions appear with the presence of the functions $|x|^{-2}$ and $|x|^{-b}$ in (INLS$_a$), related especially to the problem of equivalence of Sobolev spaces. This leads us to impose some technical restrictions on the parameters $\alpha$, $b$ and $a$, given in \eqref{conditionINLSa}. 
\begin{remark}
Observe that Theorem 1.1 also holds for $b=0$, thus we have that (NLS$_a$) is locally well-posed in $H^1(\mathbb{R}^N)$. In this particular case, we have a local result a little diferent from Luo-Miao-Murphy in \cite{Murphy1}, they showed local well posedness for $3\leq N\leq 6$, assuming \eqref{conditionNLSa}. Our result holds for any dimensions $N\geq 3$, however the condition on $a$ is weaker than Luo-Miao-Murphy'result. Note also that we improve the range of the parameter $\alpha$ to\footnote{Note that, in Theorem \ref{LWP} we have the condition $\alpha>\frac{2-2b}{N}$, however when $b=0$ we can even consider $\alpha =\frac{2}{N}$ (see Lemma \ref{Lemalocal1}).} $\alpha\geq \frac{N}{2}$. On the other hand, if $b<1$ then we have a lower bound for the parameter $\alpha$ in Theorem \ref{LWP} and if $b\geq1$ we then need $\alpha>0$.
\end{remark}


As an immediately consequence we obtain the following result.
\begin{corollary}\label{C1LWP} Assume one of the following conditions:
\begin{itemize}
\item [(i)] $N\geq3$, $1\leq b<\min\left\{\frac{N}{2},4\right\}$, $0<\alpha<\frac{4-2b}{N-2}$ and $a>-\frac{(N-2)^2}{4}+\left(\frac{\alpha(N-2)-(2-2b)}{2(\alpha+1)}\right)^2$;
\item [(ii)] $N=3$, $\alpha=2$, $0<b<1$ and $a>-\frac{1}{4}+\frac{b^2}{9}$.
\end{itemize}
If $u_0 \in H^1_a(\mathbb{R}^N)$, then the same result of Theorem \ref{LWP} holds.
\end{corollary}

\ It is worth mentioning that Corollary \ref{C1LWP} (ii) can be seen as an extension of a local result by Killip-Murphy-Visan-Zheng \cite{Murphy} to the INLS$_a$ model.

\begin{remark}\label{Energy method}
The range of the parameters $b$ and $a$ in Theorem \ref{LWP} are more restricted than in \cite{Suzuki}. That is, applying the energy method we obtain a better result than using the Kato method, however in Theorem \ref{LWP} we obtain an extra information on the solution.  
\end{remark}

\ In the sequel we establish small data global results in $H^1_a(\mathbb{R}^N)$, for $N\geq 3$ and $\tfrac{4-2b}{N}<\alpha<\frac{4-2b}{N-2}$. Since we use the Strichartz estimates (see Section \ref{sec2}) we show global well-posedness for radial and non-radial initial data. Here, $e^{-it\mathcal{L}_a}u_0$
denotes the solution to the
linear problem associated to \eqref{INLSa} and the Strichartz norm $\| \cdot \|_{S(\dot{H}^{s_c})}$ is defined in Section \ref{sec2ponto1}. 
\begin{theorem}[Radial small data theory]\label{GWP1}
Let $a>0$, $0<b<\min\{0,\frac{N}{2}\}$ and $\frac{4-2b}{N}<\alpha<\frac{4-2b}{N-2}$ $(\alpha <3-2b$ if $N=3)$. If $u_0\in H^1_a(\mathbb{R}^N)$ radial with $\|u_0\|_{H^1}\leq M$ then there exists a unique global solution $u$ of \eqref{INLSa} such that if $\|e^{-it\mathcal{L}_a}u_0\|_{S(\dot{H}^{s_c})}<\delta_{sd}$, there exists a unique global solution $u$ of \eqref{INLSa} such that
$$
\|u\|_{S(\dot{H}^{s_c})}\leq  2\|e^{-it\mathcal{L}_a}u_0\|_{S(\dot{H}^{s_c})} \quad \textnormal{and}\quad \|u\|_{S\left(L^2\right)}+\|\sqrt{\mathcal{L}_a}  u\|_{S\left(L^2\right)}\leq 2c\|u_0\|_{H^1_a},
$$
for some universal constant $c>0$.
\end{theorem}

In Theorem \ref{GWP1}, when $N= 3$, we have an extra restriction on $\alpha$, namely $\alpha <3-2b$. To reach $\alpha<4-2b$, we need to restrict the parameter $b$. This restriction comes from the need of $\alpha> 1$ in the fixed point argument. To this end, we use the norm $\|u\|_{ L^{\bar{a}}_tL^{\bar{r}}_x}$ denoted by
\begin{equation}\label{quaseadmissivel}
\|u\|_{\widetilde{S}(\dot{H}^{s_c})}=\|u\|_{ L^{\bar{a}}_tL^{\bar{r}}_x}, 
\end{equation}
where ($\bar{a},\bar{r})$ only satisfies $\bar{r} \geq 2$ and the relation of $\dot{H}^{s_c}$-admissible pair, i.e., $\frac{2}{\bar{a}}=\frac{N}{2}-\frac{N}{\bar{r}}-s_c$ and not the remaining conditions 
(see Section \ref{sec2}).  
\begin{theorem}\label{GWPN=3}
Let $N=3$, $a>0$ and $0<b<\tfrac{3}{2}$. Assume $u_0\in H^1_a(\Real^N)$ and one of the following conditions:
\begin{itemize}
\item [(i)] $\max\{1,\tfrac{4-2b}{N}\}<\alpha<4-2b$;
\item [(ii)] $\tfrac{4-2b}{N}<\alpha<4-2b$\; and\; $0<b<\frac{1}{2}$;
\item [(iii)] $3-2b\leq \alpha<4-2b$ \;and \;$0<b<1$,
\end{itemize}
then the same result as in Theorem \ref{GWP1} holds, replacing $\|\cdot \|_{ S(\dot{H}^{s_c})}$ by $\|\cdot\|_{\widetilde{S}(\dot{H}^{s_c})}$.
\end{theorem}
We remark that Theorem \ref{GWPN=3} holds for general initial data. On the other hand, Theorem \ref{GWPN=3}-(ii) shows global well posedness in the full intercritical regime, however with $b<\frac{1}{2}$.  The gap
$\frac{1}{2} \leq b < \frac{3}{2}$ is still an open problem. Moreover, in the particular case\footnote{The case $\frac{4-2b}{3}<\alpha<3-2b$ was obtained in Theorem \ref{GWP1}.}
$3-2b \leq \alpha < 4-2b$, we have a better range for $b$ than in (ii).

The next result holds for non-radial data and $a<0$, however only for dimensions $N=3,4,5$. Here, we also use the norm $\| \cdot \|_{\widetilde{S}(\dot{H}^{s_c})}$.
\begin{theorem}[Small data theory]\label{GWP2}
Let $0<b<\tfrac{6-N}{2}$ and $u_0\in H^1_a$ with $\|u_0\|_{H^1_a}\leq M$, for some $M>0$. Assume that $(N,a,\alpha)$ satisfy \begin{equation}\label{conditionINLSaG}
\begin{cases} a>-\tfrac{1}{4}\;\quad\qquad\qquad \qquad \qquad \qquad\ \textnormal{if}\;\;N=3,\;\;\;\tfrac{4-2b}{3}<\alpha\leq\; 2-2b \quad \textnormal{and} \;\; 0\leq b<\tfrac{1}{2},\\
a>-\tfrac{(N-2)^2}{4}+\left(\tfrac{\alpha(N-2)-(2-2b)}{2(\alpha+1)}\right)^2\;\;\textnormal{if}\;\;3\leq N\leq 5,\;\;\; \max\{\tfrac{4-2b}{N},\frac{2-2b}{N-2},1\}<\alpha< \tfrac{4-2b}{N-2}.
\end{cases}
\end{equation}
Then there exists $\delta_{sd}=\delta_{sd}(M)>0$ such that if $\|e^{-it\mathcal{L}_a}u_0\|_{\widetilde{S}(\dot{H}^{s_c})}<\delta_{sd}$, then there exists a unique global solution $u$ of \eqref{INLSa} such that
\begin{equation}
\|u\|_{\widetilde{S}(\dot{H}^{s_c})}\leq  2\|e^{-it\mathcal{L}_a}u_0\|_{\widetilde{S}(\dot{H}^{s_c})} \quad \textnormal{and}\quad \|u\|_{S\left(L^2\right)}+\|\sqrt{\mathcal{L}_a}  u\|_{S\left(L^2\right)}\leq 2c\|u_0\|_{H^1_a},  \end{equation}
for some universal constant $c>0$. 
\end{theorem}
\begin{remark}\label{importante}
The results above still hold, with the same proof, if one restricts the time interval to $[t_0,+\infty)$ or $(-\infty,t_0]$, instead of $\mathbb{R}$, where $u(t_0)=u_0$. By time-translation invariance, we assume
$t_0=0$ in Theorem \ref{GWP2}.
\end{remark}

As mentioned above the main tool to show the local and global well-posedness is the Fixed Point Theorem, which is based on the Strichartz estimates. Similarly as in the local theory the main difficulty here is to look for admissible pairs to establish the equivalence of Sobolev spaces, mainly when $a<0$. 

Once global results are proved, the natural route is to study the asymptotic behavior of such global solutions as $t\rightarrow \pm  \infty$. We show that our solutions scatter to a solution of the linear problem in $H^1_a(\Real^N)$. In addition, we construct the wave operator associated to eq. \eqref{INLSa}. This is the reciprocal problem of the scattering theory, which consists in constructing a solution with a
prescribed scattering state.


\begin{theorem}\label{Wave operator} 
Assume the assumptions in Theorems \ref{GWP1}, \ref{GWPN=3} and \ref{GWP2}. 
\begin{itemize}
\item [(i)] Let $u(t)$ be  a global solution of \eqref{INLSa} with initial data $u_0 \in  H^1_a(\mathbb{R}^N)$. If $\sup\limits_{t\in \mathbb{R}}\|u(t)\|_{H^1_a}\leq M$ and 
$\|u\|_{S(\dot{H}^{s_c})}< +\infty\;\; (\textnormal{or}\; \|u\|_{\widetilde{S}(\dot{H}^{s_c})}< +\infty)$, then $u(t)$ scatters in $H^1_a$, that is, there exists  $\phi^{\pm}\in H^1_a$ such that
$$
\lim_{t\rightarrow \pm\infty}\|u(t)-e^{-it\mathcal{L}_a}\phi^{\pm}\|_{H^1_a}=0.
$$
\item [(ii)] For any $\phi \in H^1_a\left(\mathbb{R}^N\right)$, there exist $T_0>0$ and
$u\in C([T_0,\infty) : H^1_a(\mathbb{R}^N))$ solution of \eqref{INLSa} satisfying
$$
\lim_{t\rightarrow \;\infty}\|u(t)-e^{-it\mathcal{L}_a}\phi\|_{H^1_a}=0.
$$ 
The analogous statement holds backward in time.
\end{itemize}
\end{theorem}


The rest of the paper is organized as follows. In Section \ref{sec2}, we introduce some notations and give a review of the Strichartz estimates. In Section \ref{Sec3} we discuss the existence of a ground state and establish global existence as well as blow up in $H^1_a(\mathbb{R}^N)$, for $L^2$-critical and intercritical cases. In Section \ref{Sec4} we study the local and global well-posedness applying the contraction mapping principle. Finally, in Section \ref{Sec5} we prove Theorem \ref{Wave operator}.

\ 

\section{Notation and Preliminaries}\label{sec2}

In this section, we introduce the notation used throughout the paper and list some useful results.  We use $C$ to denote various constants that may vary line by line. If $a$ and $b$ be positive real numbers, the
notation $a \lesssim b$ means that there exists a positive constant $C>0$ such that\footnote{The constant $C$ may depend on parameters, such as the dimension $N$, as well on \textit{a priori} estimates on the solution, but never on the solution itself or on time} $a \leq Cb$. The notation $a \sim b$ means $a\lesssim b$ and $b \lesssim a$. Given a real number $r$, we use $r^\pm$ to denote $r\pm\varepsilon$ for some $\varepsilon>0$ sufficiently small. For a subset $A\subset \mathbb{R}^N$, its complement is denoted by $A^C=\mathbb{R}^N \backslash A$ and the characteristic function $\chi_A$ denotes the function
that has value $1$ at points of $A$ and $0$ at points of $A^C$. Given $x,y \in \mathbb{R}^N$, $x \cdot y$ denotes the usual inner product of $x$ and $y$ in $\mathbb{R}^N$.

The norm in the Sobolev spaces $\dot{H}^{s,r}=\dot{H}^{s,r}(\mathbb{R}^N)$ and $H^{s,r}=H^{s,r}(\mathbb{R}^N)$, are defined by
\begin{equation}
\|f\|_{\dot{H}^{s,r}}:=\|D^sf\|_{L^r} \quad \textnormal{and}\quad \|f\|_{H^{s,r}}:=\| \langle D \rangle^s f \|_{L^r},
\end{equation}
where $D^s f:= \sqrt{-\Delta}^s f=(|\xi|^s\widehat{f})^{\vee}
$  and $\langle \cdot \rangle=(1+|\cdot|^2)^{\frac{1}{2}}$. 
If $r=2$ we denote $H^{s,2}$ and $\dot{H}^{s,2}$ simply by $H^s$ and  $\dot{H}^{s}$, respectively. Similarly, we define Sobolev spaces  $\dot{H}_a^{s,r}$ and $H_a^{s,r}$ associated to $\mathcal{L}_a$ by the closure of $\mathbb{C}^\infty_0(\mathbb{R}^N \backslash \{0\} )$ under the norms
\begin{equation}
\|u\|_{\dot{H}_a^{s,r}}:=\|\sqrt{\mathcal{L}_a}^s u\|_{L^r}    \quad\textnormal{and} \quad \|u\|_{H_a^{s,r}}:=\| \langle\sqrt{\mathcal{L}_a}\rangle^s u\|_{L^r}.
\end{equation}
We abbreviate $\dot{H}_a^s(\mathbb{R}^N)= \dot{H}_a^{s,2}(\mathbb{R}^N))$ and $H_a^s(\mathbb{R}^N ) = H_a^{s,2}(\mathbb{R}^N )$. Note that, by the sharp Hardy inequality, one has
\begin{equation}
\|u\|_{\dot{H}_a^{1}} \sim \|u\|_{\dot{H}^{1}} \quad \textnormal{for} \quad a>-\left(\frac{N-2}{2}\right)^2.
\end{equation}

We also define, for $1\leq p<\infty$, the weighted Sobolev space $L^p_b=L^p_b(\Real^N)= \{ f: \;\|f\|_{p,b} < +\infty\}$, where
\begin{equation}
\|f\|_{p,b} = \left[\int |x|^{-b}|f(x)|^{p}\, dx\right]^\frac{1}{p}.
\end{equation}

Let $ q,r >0$, $s\in \mathbb{R}$, and $I\subset \mathbb{R}$ an interval; the mixed norms in the spaces $L^q_{I}L^r_x$ and $L^q_{I} H^s_x$ of a function $f=f(t,x)$ are defined as
$$
\|f\|_{L^q_{I}L^r_x}=\left(\int_I\|f(t,\cdot)\|^q_{L^r_x}dt\right)^{\frac{1}{q}}
\qquad
\mbox{and}
\qquad
\|f\|_{L^q_{I}H^s_x}=\left(\int_I\|f(t,\cdot)\|^q_{H^s_x}dt\right)^{\frac{1}{q}},
$$
with the usual modifications if either $q=\infty$ or $r=\infty$. When the  space 99+-6integration is restricted to a subset $A\subset\mathbb{R}^N$ then the mixed norm will be denoted by $\|f\|_{L_I^qL^r_x(A)}$. Moreover, if $I=\mathbb{R}$ we shall use the notations $\|f\|_{L_t^qL^r_x}$ and $\|f\|_{L_t^qH^s_x}$.

Next, we recall some important inequalities. 
To state the estimates below, it is useful to
introduce the parameter
\begin{equation}\label{ro}
\rho=\frac{(N-2)-\sqrt{(N-2)^2+4a}}{2}. \end{equation}
\begin{lemma}[\bf Equivalence of Sobolev spaces] Fix $N \geq 3$, $a \geq -(\frac{N-2}{2})^2$, and $0<s<2$. If $1<p< \infty$ satisfies $\frac{s+\rho}{N}
 <\frac{1}{p}< \min\{1,
\frac{N-\rho}{N}\}$, then
\begin{equation}
\|D^s f\|_{L^p}\lesssim \| \mathcal{L}_a^{\frac{s}{2}} f\|_{L^p}\;\textnormal{for all}\;f\in \mathbb{C}^\infty_0(\mathbb{R}^N \backslash \{0\} ).    
\end{equation}
If $\max\{\frac{s}{N},\frac{\rho}{N} \}
 <\frac{1}{p}< \min\{1,
\frac{N-\rho}{N}\}$, then
\begin{equation}
\|\mathcal{L}_a^{\frac{s}{2}} f\|_{L^p}\lesssim \|  D^sf\|_{L^p}\;\textnormal{for all}\;f\in \mathbb{C}^\infty_0(\mathbb{R}^N \backslash \{0\} ).    
\end{equation}
\begin{proof}
See \cite{Killip-Miao-Visan-Zhang-Zheng}
\end{proof}
\end{lemma}
\begin{remark}\label{equivalence}
Let $0 <s< 2$. It is easy to see that, if $a>0$ then $\|f\|_{H_a^{s,r}} \sim \|f\|_{H^{s,r}}$, provided that $1<r<\frac{N}{s}$. When $-\frac{(N-2)^2}{4}\leq a<0$ we have $0<\rho<\frac{(N-2)}{2}$ and so  $\|f\|_{H_a^{s,r}}  \sim \|f\|_{H^{s,r}}$ if $\frac{N}{N-\rho}<r<\frac{N}{s+\rho}$.  
\end{remark}

\begin{lemma}\label{strauss_lemma} If $f \in H^1(\mathbb{R}^N)$ is radial, $N\geq 2$, then, for any $R>0$,
\begin{equation}\label{Strauss}
	\|f\|_{L^\infty_{\left\{|x|\geq R\right\}}} \lesssim R^{-\frac{N-1}{2}}\|f\|_{L^2}^{\frac{1}{2}}\|\nabla f\|_{L^2}^{\frac{1}{2}}.
\end{equation}
\begin{proof}
See Strauss \cite{Strauss}.
\end{proof}
\end{lemma}

The next lemma implies that the Strichartz estimates (Lemma \ref{Strichartzpotential}) hold.
\begin{lemma}[\bf Dispersive estimate]\label{Dispersive equ}
Let f be a radial function.
\begin{itemize}
\item [(i)] If $a \geq 0$, then we have
\begin{equation}\label{DE1}
\|e^{it\mathcal{L}_a}f\|_{L^\infty}\lesssim |t|^{-\frac{N}{2}}\|f\|_{L^{1}}.
\end{equation}
\item [(ii)] If $-\frac{(N-2)^2}{4}< a <0$, then
\begin{equation}\label{DE2}
\|(1+|x|^{-\rho})^{-1}e^{it\mathcal{L}_a}f\|_{L^\infty}\lesssim \frac{1+|t|^\rho}{|t|^{\frac{N}{2}}}\|(1+|x|^{-\rho})f\|_{L^{1}},
\end{equation}
with $\rho$ being as in \eqref{ro}.
\end{itemize}
\begin{proof} See Zheng \cite{Zheng}.
\end{proof}
\end{lemma}

\subsection{ Strichartz-Type Estimates}\label{sec2ponto1}
Before stating the Strichartz estimates, we need the following definitions.

We say the pair $(q, r)$ is Schr\"odinger admissible (S-admissible or $L^2$-admissible for short) if it satisfies
\begin{equation}\label{PA0}
\frac{2}{q}=\frac{N}{2}-\frac{N}{r}
\end{equation}
where
\begin{equation}\label{CPA0}
\begin{cases}
2 \leq  r  \leq \frac{2N}{N-2}\;\;\textnormal{if}\;\;\;  N\geq 3,\\
2 \leq  r < + \infty\;  \hspace{0.2cm}\textnormal{if}\;\;\; N=1,2.
\end{cases}
\end{equation}
Also, given a real number $s>0$, the pair $(q,r)$ is called $\dot{H}^s$-admissible if\footnote{It is worth mentioning that the pair $\left(\infty,\frac{2N}{N-2s_c}\right)$
also satisfies the relation \eqref{CPA1}, however, in our work we will not make use of
this pair when we estimate the nonlinearity. See Section \ref{Sec5}.} 
\begin{equation}\label{CPA1}
\frac{2}{q}=\frac{N}{2}-\frac{N}{r}-s
\end{equation}
with
\begin{equation}\label{HsAdmissivel}
\begin{cases}
\frac{2N}{N-2s} \leq  r  <\frac{2N}{N-2}\;\textnormal{if}\;\; \; N\geq 3,\\
\frac{2N}{N-2s} \leq  r < + \infty\;  \hspace{0.1cm}\textnormal{if}\;\;\; N=1,2.
\end{cases}
\end{equation}
We set\footnote{The restriction for $(q, r)$ $\dot{H}^0$-admissible is given by \eqref{CPA0}.} $\mathcal{A}_s:=\{(q,r);\; (q,r)\; \textnormal{is} \;\dot{H}^s\textnormal{-admissible}\}$. Also, given $(q,r)\in \mathcal{A}_s$, by $(q',r')$ we denote its dual pair, that is, $\frac{1}{q}+\frac{1}{q'}=1$ and $\frac{1}{r}+\frac{1}{r'}=1$. We define the Strichartz norm by
$$
\|u\|_{S(\dot{H}^{s})}=\sup_{(q,r)\in \mathcal{A}_{s}}\|u\|_{L^q_tL^r_x} $$
and the dual Strichartz norm by
$$
\|u\|_{S'(\dot{H}^{-s})}=\inf_{(q,r)\in \mathcal{A}_{-s}}\|u\|_{L^{q'}_tL^{r'}_x}.
$$
If $s=0$ then $\mathcal{A}_0$ is the set of all $S$-admissible pairs. We denote $S(\dot{H}^0)$ by $S(L
^2)$. We write $S(\dot{H}^s)$ or $S'(\dot{H}^{-s})$ if the mixed norm is evaluated over $\mathbb{R}\times\mathbb{R}^N$. To indicate the restriction to a time interval $I\subset (-\infty,\infty)$ or a subset $A\subset\mathbb{R}^N$, we will use the notations $S(\dot{H}^s(A);I)$ and $S'(\dot{H}^{-s}(A);I)$.\\

Finally, we define the  norm\footnote{It was mentioned in the introduction, see \eqref{quaseadmissivel}.} $\|u\|_{\widetilde{S}(\dot{H}^{s_c})}=\|u\|_{ L^{\bar{a}}_tL^{\bar{r}}_x}$, where ($\bar{a},\bar{r})$ only satisfies $\bar{r} \geq 2$ and the relation \eqref{CPA1} but not necessarily \eqref{HsAdmissivel}. Specifically, the number $\bar{r}$ does not need to satisfy the condition $\bar{r}<\frac{2N}{N-2}$. 

\ We end this section by recalling the Strichartz estimates for the linear flow $e^{-it \mathcal{L}_a}$. They were first proved by Burq-Planchon-Stalker-Tahvildar-Zadeh\footnote{They showed Strichartz estimates for $e^{-it \mathcal{L}_a}$  except the endpoint $(q,r)=(2,\frac{2N}{N-2})$.} in \cite{Burq}. Zhang-Zheng \cite{Zhang2017} confirmed the double endpoint case.
\begin{lemma}\label{Strichartzpotential0}
If $N\geq 3$ and $a>-\frac{(N-2)^2}{4}$. Then,
\begin{eqnarray}
\| e^{-it \mathcal{L}_a}f \|_{S(L^2;I)} &\lesssim & \|f\|_{L^2},\label{SE1} \\
\| e^{-it \mathcal{L}_a}f \|_{S(\dot{H}^s;I)} &\lesssim &  \|f\|_{\dot{H}^s_a} \label{SE2}\\
\left \| \int_{t_0}^t e^{-i(t-t')\mathcal{L}_a}g(\cdot,t') dt' \right \|_{S(L^2;I) } &\lesssim& \|g\|_{S'(L^2;I)} \label{SE3}
\end{eqnarray} 
\begin{proof}
See \cite{Burq} and \cite{Zhang}.
\end{proof}
\end{lemma}

\begin{lemma}\label{Strichartzpotential}
If $N\geq 3$, $a\geq 0$ and $g$ a radial function. Then,
\begin{equation}\label{SE5}
\left \| \int_{t_0}^t e^{-i(t-t')\mathcal{L}_a}g(\cdot,t') dt' \right \|_{S(\dot{H}^s;I) } \lesssim \|g\|_{S'(\dot{H}^{-s};I)},
\end{equation}
where $I\subset\mathbb{R}$ be an interval and $t_0\in I$.
\begin{proof}
The estimate is readily obtained from the main result in Foschi \cite{Foschi}, given the $L^1 \rightarrow L^{\infty}$ estimate in Lemma \ref{Dispersive equ}, and the invariance of the $L^2$ norm by $e^{-it\mathcal{L}_a}$.
\end{proof}
\end{lemma}
\begin{remark}
As usual, if $I=(T,+\infty)$ then in Lemma \ref{Strichartzpotential} one may replace the integral $\int_{t_0}^t$ by $\int_t^{+\infty}$. A similar statement holds if $I=(-\infty, T)$.
\end{remark}
\begin{remark}\label{equivalencesobolevspaces}
In the case, when $s=0$, we have the norms $\|u\|_{S(L^2)}=\sup_{(q,r)\in \mathcal{A}_{0}}\|u\|_{L^q_tL^r_x}$ and $\|u\|_{S'(L^2)}=\inf_{(q,r)\in \mathcal{A}_{0}}\|u\|_{L^{q'}_tL^{r'}_x}$.
\end{remark}

\ 

\section{Global well-posedness and blow-up in $H^1_a$}\label{Sec3}

In this section, we prove results about the ground state $Q$, together with a dichotomy between global existence and finite-time blow-up below a mass-energy threshold. We first show the existence of a ground state (Proposition \ref{exist_ground}).

\subsection{Existence of a ground state}

We start by proving a compact embedding result.  
\begin{lemma}[Compactness of an immersion]\label{compacidade}
If $N \geq 3$, $a > -\tfrac{(N-2)^2}{4}$, $0 < b < 2$ and $0 < \alpha < \tfrac{4-2b}{N-2}$, then $H^1_a(\mathbb{R}^N)$ is compactly embedded in $L^{\alpha+2}_b(\mathbb{R}^N)$.
\end{lemma}

\begin{proof}
Let $\{f_n\}_n$ be a bounded sequence in $H^1_a$. Since $a>-\frac{(N-2)^2}{4}$, Hardy's inequality yields that $\{f_n\}_n$ is also bounded in the standard space $H^1$, so we may assume that $f_n \rightharpoonup f \in H^1$ weakly in $H^1$. Now, for $R, \epsilon > 0$,
\begin{equation}
\int |x|^{-b}|f_n-f|^{\alpha+2}dx \lesssim \left(\int_{|x|\leq R}|x|^{-\frac{bN}{b+\epsilon}}dx \right)^{\frac{b+\epsilon}{N}}\|f_n-f\|_{L^\frac{N(\alpha+2)}{N-b-\epsilon}}^{\alpha+2}+\frac{1}{R^b}(\|f_n\|_{L^{\alpha+2}} + \|f\|_{L^{\alpha+2}})^{\alpha+2}.
\end{equation}

Thus, by Sobolev and Rellich-Kondrachov, since
\begin{equation}
    2 < \alpha + 2 < \frac{N(\alpha+2)}{N-b-\epsilon} < \frac{2N}{N-2},
\end{equation}
if $0 < \alpha < \tfrac{4-2b}{N-2}$ and $\epsilon$ is small, the result follows.\end{proof}

\begin{lemma}[Adapted Gagliardo-Nirenberg inequality]\label{lem_gn} Let $N\geq 3$. If $f \in H^1_a$, $0<b<2$ and $0 < \alpha < \frac{4-2b}{N-2}$, then
\begin{equation}
\int |x|^{-b}|f|^{\alpha+2}\,dx \leq C_a \|\sqrt{\mathcal{L}_a} f\|_{L^2}^{\tfrac{N\alpha+2b}{2}} \| f\|_{L^2}^{\tfrac{4-2b-\alpha(N-2)}{2}}.
\label{gagliardo}
\end{equation}
Equality in the bound above is attained by a function $Q \in H^1_a$, which  is a positive solution to the elliptic equation

\begin{equation}\label{eq_ellip_ground}
    \mathcal{L}_a Q - Q + |x|^{-b}|Q|^{\alpha}Q = 0.
\end{equation}
\end{lemma}

\begin{proof}
We mimic the classic proof for $a=b=0$, and exploit the compactness given by the immersion $H^1_a \hookrightarrow L^{\alpha+2}_b$. Let $\{f_n\}_n \subset H^1_a$ be a minimizing sequence for the Weinstein functional
\begin{equation}
J_a(f) = \frac{ \|\displaystyle\sqrt{\mathcal{L}_a} f\|_{L^2}^{\tfrac{N\alpha+2b}{2}} \| f\|_{L^2}^{\tfrac{4-2b-\alpha(N-2)}{2}}}{\displaystyle\int |x|^{-b}|f|^{\alpha+2}\,dx}.
\end{equation}

By choosing $\lambda_n >0$ and $\mu_n >0$ such that $g_n(x) = \lambda_n f_n( \mu_n x)$ satisfy $\|\sqrt{\mathcal{L}_a}  g_n\|_{L^2} = \| g_n\|_{L^2}=1$, and noting that $J_a(f_n) = J_a(g_n)$ for all $n$, we may assume, by compactness, that $\{g_n\}_n$ converges strongly to $g$ in $L^{\alpha+2}_b$ (see Lemma \ref{compacidade}) and by reflexiveness, weakly in $H^1_a$. Moreover, since $|\nabla|f|| \leq |\nabla f|$ for all $f \in H^1_a$, we can assume that $g \geq 0$. Furthermore, by Hölder, Hardy and Sobolev, we see that $J^*_a = \displaystyle\inf_{f \in H^1_a\backslash \{0\}} J(f) > 0$, and since $\{g_n\}_n$ is a minimizing sequence,
\begin{equation}
    \int |x|^{-b}|g(x)|^{\alpha+2} \, dx = \frac{1}{J^*_a} \in (0,+\infty).
\end{equation}
This shows that $g \neq 0$. Now note that $\|\sqrt{\mathcal{L}_a}  g\|_L^2 = \| g\|_L^2 =1$, since otherwise it would contradict the minimality of $J_a(g) = J^*_a$. Defining $C_a := (J_a^*)^{-1}$, one sees that \eqref{gagliardo} holds.

The Euler-Lagrange equation for $g$ gives
\begin{equation}
J_a^*\tfrac{N\alpha+2b}{2} \mathcal{L}_a g + J_a^*[\tfrac{4-2b-\alpha(N-2)}{2}] g - (\alpha+2)g^{\alpha+1} = 0.
\end{equation}

Finally, defining $g(x) = \lambda Q(\mu x)$, with 
\begin{equation}
\lambda = \left\{ \tfrac{[4-2b-\alpha(N-2)]J^*_a}{2(\alpha+2)}\right\}^{\frac{1}{\alpha}}, 
\end{equation}
and
\begin{equation}
\mu = \tfrac{4-2b-\alpha(N-2)}{N\alpha+2b}, 
\end{equation}
one sees that $Q$ solves \eqref{eq_ellip_ground}.
\end{proof}

In the following lemma we obtain Pohozaev-type identities which are satisfied by any solution of \eqref{eq_ellip_ground}. The proof follows 
 multiplying \eqref{eq_ellip_ground} by Q and $x \cdot \nabla Q$ and using integration by parts. We omit the details. 
\begin{lemma}[Pohozaev identities]\label{lem_poho}
If $Q \in H^1_a$ is a solution to \eqref{eq_ellip_ground}, then the following identities hold

\begin{align}
&\int \left|\sqrt{\mathcal{L}_a} Q\right|^{2}\,dx= \left(\tfrac{N\alpha+2b}{4-2b-\alpha(N-2)}\right)\|Q\|_{L^2}^2,
\label{pohozaev}\\
&\int |x|^{-b}|Q|^{\alpha+2}\,dx = \left(\tfrac{2(\alpha + 2)}{4-2b-\alpha(N-2)}\right) \|Q\|_{L^2}^2.
\label{pohozaev2}
\end{align}
In particular, one can write the mass (and therefore the quantities in \eqref{pohozaev} and \eqref{pohozaev2}) in terms of the sharp constant $C_a$. Namely,
\begin{equation}
M[Q] = \left\{\tfrac{2(\alpha+2)}{N\alpha+2b} \tfrac{[4-2b-\alpha(N-2)]^{\tfrac{N\alpha-(4-2b)}{4}}}{C_a}\right\}^\frac{1}{\alpha+2}.
\end{equation}
\end{lemma}

With the previous results we show Proposition \ref{exist_ground}.
\begin{proof}[\bf{Proof of Proposition \ref{exist_ground}}]
The existence part of Proposition \ref{exist_ground} follows from Lemma \ref{lem_gn}, and the uniqueness of the mass, $\dot{H}^1_a$ and energy follows from Lemma \ref{lem_poho}.
\end{proof}

\subsection{Global behavior in the mass-critical case}

We now study the global
existence and blow-up in $H^1_a$ of \eqref{INLSa}, when $\alpha=\tfrac{4-2b}{N}$ and  $\lambda=1$. We start by proving the global well-posedness. 
\subsubsection{Global well-posedness}
\begin{proof}[\bf Proof of Theorem \ref{GBCrit}a)]
We make use here of the sharp Gagliardo-Nirenberg-type inequality (Lemma \ref{lem_gn}) in the case $\alpha = \tfrac{4-2b}{N}$:

\begin{equation}
    \int |x|^{-b}|f|^{\frac{4-2b}{N}+2}\,dx \leq \tfrac{2-2b+N}{N} \|\sqrt{\mathcal{L}_a} f\|_{L^2}^{2} \left(\frac{\| f\|_{L^2}}{\| Q\|_{L^2}}\right)^{\tfrac{4-2b}{N}}.
\end{equation}

By energy conservation:
\begin{align}
E[u_0] &= \frac{1}{2}\|\sqrt{\mathcal{L}_a}u(t)\|_{L^2}^2 - \frac{N}{4-2b+2N}\int |x|^{-b}|u(t)|^{\frac{4-2b}{N}+2}\,dx \\
&\geq \frac{1}{2}\|\sqrt{\mathcal{L}_a}u(t)\|_{L^2}^2 \left[1-\left(\frac{\|u_0\|_{L^2}}{\| Q\|_{L^2}}\right)^{\tfrac{4-2b}{N}}\right].
\end{align}
\ 
Therefore, if $\|u_0\|_{L^2} < \|Q\|_{L^2}$, the corresponding solution $u$ extends globally in time.
\end{proof}
\subsubsection{Blow-up} The main tool to prove the blow-up results is the Virial identity.
\begin{lemma}[Virial identity]\label{Morawetz}
Let $\varphi: \mathbb{R}^N \rightarrow \mathbb{R}$ be a real weight. Define
\begin{equation}
V(t) = \int \varphi |u(t)|^2dx.
\end{equation}
Then, if $u$ is a solution to \eqref{INLS}, we have the following identities  
$$
V'(t) = 2 \Im\int \bar{u} \nabla u \cdot \nabla \varphi \, dx,
$$

\begin{align}
V''(t) &= \left(\frac{4}{\alpha+2}-2\right)\int|x|^{-b}|u|^{\alpha+2} \Delta \varphi  dx-\frac{4b}{\alpha+2}\int|x|^{-b-2}|u|^{\alpha+2}x\cdot\nabla \varphi dx\label{main_1}\\
&\quad + 4\Re\sum_{i,j}\int \varphi_{ij}\bar{u}_i u_j dx+4a\int\frac{ x\cdot\nabla \varphi}{|x|^4}|u|^2 dx\label{main_2}\\
&\quad -\int|u|^2\Delta\Delta \varphi dx \label{err_mass}. 
\end{align}
\begin{proof}
The proof is considered standard, see for instance \cite{paper3,Murphy}.
\end{proof}
\end{lemma}

\begin{proof}[{\bf Proof of Theorem \ref{GBCrit}b)}]

The blow-up for fast-decaying, negative-energy solutions is proved using Glassey's argument, with the Virial identity by taking $\phi=|x|^2$. It means,

\begin{equation}
   \partial^2_{tt}\int |x|^2|u(t)|^2 dx= 16E[u_0] < 0.
\end{equation}

We now show for radial solutions. Define $\psi: (0,+\infty) \to \mathbb{R}$ as
\begin{equation}
    \psi(r) = \begin{cases}
    r^2, & r \leq 1,\\
    r^2-\frac{(r-1)^4}{2}, & 1 < r \leq 2\\
    \frac{7}{2}, &r > 2.
    \end{cases}
\end{equation}
Note that $\psi$ and all of its (weak) derivatives are essentially bounded. Define also, for $R>0$, $\phi_R(x) = R^2 \psi(x/R)$ and 
\begin{equation}
V_R(t) = \int \phi_R |u(t)|^2dx.
\end{equation}
By Lemma \ref{Morawetz} and the radiality of $u$, in the case $\alpha = \tfrac{4-2b}{N}$, we have
\begin{align}
V_R''(t) &= 16E[u(t)] +\frac{4-2b}{2-b+N}\int(2N-\Delta \phi_R) |x|^{-b}|u|^{\tfrac{4-2b}{N}+2}dx\\  &\quad+\frac{2Nb}{2-b+N}\int(2-\frac{x\cdot\nabla \phi_R}{|x|^2}) |x|^{-b}|u|^{\tfrac{4-2b}{N}+2}dx\\
    &\quad- 4\int (2-\psi''(|x|/R))|\nabla u|^2dx -4a\int(2-\frac{ x\cdot\nabla \phi_R}{|x|^2})\frac{1}{|x|^2}|u|^2dx -\int \Delta\Delta \phi_R |u|^2dx\\
    & \leq 16E[u_0] - 4\int (2-\psi''(|x|/R))|\nabla u|^2 dx + O(\frac{1}{R^b}\int_{|x|\geq R} [\eta_R(x) |u|]^{\tfrac{4-2b}{N}}|u|^2dx )+ O(\frac{1}{R^2} \int|u|^2dx),
\end{align}

where 
\begin{equation}
\eta_R(x)=\left\{\frac{1}{2-b+N}\left[(4-2b)(2N-\Delta\phi_R(x))+2Nb\left(2-\frac{x\cdot\nabla \phi_R(x)}{|x|^2}\right)\right]\right\}^\frac{N}{4-2b}.
\end{equation}
By mass conservation, we are left to control the third term in the last inequality, if $R>0$ is large enough. We use Strauss and Young inequalities:
\begin{align}
 \frac{1}{R^b}\int_{|x|\geq R} [\eta_R(x) |u|]^{\tfrac{4-2b}{N}}|u|^2dx &\lesssim \frac{1}{R^b}\|\eta_R u\|_{L^{\infty}_{\{|x|\geq R\}}}^{\tfrac{4-2b}{N}} \int |u|^2dx\\ & \lesssim \frac{1}{R^{\frac{(2-b)(N-1)+Nb}{N}}}(\|\nabla(\eta_R u)\|_{L^2}^{\frac{2-b}{N}}\|u\|^{\frac{2-b}{N}}_{L^2})\int |u|^2dx\\ &\leq \epsilon \|\eta_R \nabla u\|_{L^2}^2 + \frac{C(\epsilon)}{R^2}\|u\|_{L^2}^2.
\end{align}

We therefore obtain

\begin{equation}
    V''_R(t) \leq 16E[u_0]-\int \left[4 (2-\psi''(|x|/R))-\epsilon\eta_R^2(x))\right]|\nabla u|^2 + O(\frac{C(\epsilon)}{R^{2}}).
\end{equation}

By the definition of $\psi$ and $\eta_R$, one can choose $\epsilon>0$ such that, almost everywhere and independently on $R$,
\begin{equation}
    4 (2-\psi''(\cdot/R))-\epsilon\eta_R^2) \geq 0.
\end{equation}

Choosing, afterwards, $R>0$ large enough, we have, 
\begin{equation}
     V''_R(t) \leq 15E[u_0] < 0,
\end{equation}
which implies finite-time blow-up.

\end{proof}
\subsection{Global behavior for the intercritical case}\hfill\\

We now state some coercivity-type (also known as \textit{energy-trapping}) results for the INLS$_a$, which are necessary for Theorem \ref{GBInter}.

\begin{lemma}\label{lem_coerc_1}
Let $N \geq 3$, $a$, $\alpha$ and $b$ as in Theorem \ref{GBInter}, and $f \in H^1_a(\mathbb{R}^N)$. Assume that, for some $\delta_0 > 0$,
\begin{equation}
M(f)^{\tfrac{1-s_c}{s_c}} E_a(f) \leq (1-\delta_0) M(Q)^{\tfrac{1-s_c}{s_c}}E_a(Q), 
\end{equation}
then there exists $\delta = \delta(\delta_0, N, \alpha, Q, a, b)$ such that
\begin{equation}\label{lem_trap_1_concl}
\left|\|f \|^{\frac{1-s_c}{s_c}}_{L^2}\|\sqrt{\mathcal{L}_a} f \|_{L^2}-\|Q\|^{\frac{1-s_c}{s_c}}_{L^2} \|\sqrt{\mathcal{L}_a} Q\|_{L^2} \right| \geq \delta \|Q\|^{\frac{1-s_c}{s_c}}_{L^2} \|\sqrt{\mathcal{L}_a} Q\|_{L^2}.
\end{equation}
In particular, if 
\begin{equation}\label{lem_trap_below}
    \|f\|^{\frac{1-s_c}{s_c}}_{L^2} \|\sqrt{\mathcal{L}_a} f\|_{L^2} < \|Q\|^{\frac{1-s_c}{s_c}}_{L^2} \|\sqrt{\mathcal{L}_a} Q\|_{L^2},
\end{equation}
then 
\begin{equation}
    \|f\|^{\frac{1-s_c}{s_c}}_{L^2} \|\sqrt{\mathcal{L}_a} f\|_{L^2} \leq (1-\delta) \|Q\|^{\frac{1-s_c}{s_c}}_{L^2} \|\sqrt{\mathcal{L}_a} Q\|_{L^2}.
\end{equation}
Similarly, if
\begin{equation}\label{lem_trap_above}
    \|f\|^{\frac{1-s_c}{s_c}}_{L^2} \|\sqrt{\mathcal{L}_a} f\|_{L^2} > \|Q\|^{\frac{1-s_c}{s_c}}_{L^2} \|\sqrt{\mathcal{L}_a} Q\|_{L^2},
\end{equation}
then 
\begin{equation}
    \|f\|^{\frac{1-s_c}{s_c}}_{L^2} \|\sqrt{\mathcal{L}_a} f\|_{L^2} \geq (1+\delta) \|Q\|^{\frac{1-s_c}{s_c}}_{L^2} \|\sqrt{\mathcal{L}_a} Q\|_{L^2}.
\end{equation}
\end{lemma}

\begin{proof}
Using Lemma \ref{lem_gn}, we write

\begin{align}
(1-\delta_0)M[Q]^\frac{1-s_c}{s_c}E_a[Q] &\geq  M[f]^\frac{1-s_c}{s_c}E_a[f] \\
&= \frac{M[f]^\frac{1-s_c}{s_c} }{2}\int \left| \sqrt{\mathcal{L}_a} f\right|^2 dx - \frac{M[f]^\frac{1-s_c}{s_c}}{\alpha+2}\int |x|^{-b}|f|^{\alpha+2}dx\\
&\geq \frac{1}{2}\left(\|f\|_{L^2}^\frac{1-s_c}{s_c} \| \sqrt{\mathcal{L}_a}f\|_{L^2}\right)^2 - \frac{C_a}{\alpha+2}\left(\|f\|_{L^2}^\frac{1-s_c}{s_c} \| \sqrt{\mathcal{L}_a}f\|_{L^2}\right)^{\frac{N\alpha+2b}{2}}.\\
\end{align}

Let $P(y) = \frac{1}{2}y^2 - \frac{C_a}{\alpha+2}y^{\frac{N\alpha+2b}{2}}$ and $y(f) = \|f\|_{L^2}^\frac{1-s_c}{s_c} \| \sqrt{\mathcal{L}_a}f\|_{L^2}$. By Lemmas \ref{lem_gn} and \ref{lem_poho}, we see that, if $f = Q$, then $y^* := y(Q)$ satisfies $P(y^*) = M[Q]^\frac{1-s_c}{s_c}E_a[Q]$. Since $P$ is continuous, increasing if $y < y^*$, and decreasing if $y>y^*$, we conclude that there exists $\delta_0 > 0$ such that.
\begin{equation}
P(y) \leq (1-\delta_0) P(y^*) \implies |y-y^*| \geq \delta y^*.
\end{equation}
By continuity, we conclude \eqref{lem_trap_1_concl}.
\end{proof}

\begin{lemma}\label{lem_coerc_2}
Under the conditions \eqref{lem_trap_1_concl} and \eqref{lem_trap_above} of the previous lemma, one also has, for some $\eta = \eta(\delta_0, Q, M(f)) > 0$ and $\epsilon = \epsilon(\delta_0, Q)>0$,
\begin{equation}\label{coerc_blow}
(1+\epsilon)\int\left|\sqrt{\mathcal{L}_a} f\right|^2dx + \left(\frac{N-b}{\alpha+2}-\frac{N}{2}\right)\int|x|^{-b}|f|^{\alpha+2}dx \leq -\eta < 0
\end{equation}
\begin{proof}
Under the notation of the proof of the previous lemma, if $$M(f)^{\tfrac{1-s_c}{s_c}} E_a(f) \leq (1-\delta_0) M(Q)^{\tfrac{1-s_c}{s_c}}E_a(Q)$$  
and
\begin{equation}
    \|f\|^{\frac{1-s_c}{s_c}}_{L^2} \|\sqrt{\mathcal{L}_a} f\|_{L^2} \geq  \|Q\|^{\frac{1-s_c}{s_c}}_{L^2} \|\sqrt{\mathcal{L}_a} Q\|_{L^2},
\end{equation}
we write
\begin{align}\label{ident_second_virial}
(1+\epsilon)\int\left|\sqrt{\mathcal{L}_a} f\right|^2 + &\left(\frac{N-b}{\alpha+2}-\frac{N}{2}\right)\int|x|^{-b}|f|^{\alpha+2} \\&=
\frac{N\alpha+2b}{2}E_a[f]-\left[\frac{N}{2}\left(\alpha-\frac{4-2b}{N}\right)-\epsilon\right]\int\left|\sqrt{\mathcal{L}_a} f\right|^2\\
&\leq (1-\delta_0)\frac{N\alpha+2b}{2}M[f]^{-\frac{1-s_c}{s_c}}M[Q]^{\frac{1-s_c}{s_c}}E_a[Q] \\&\quad\quad\quad\quad\quad-\left[\frac{N}{2}\left(\alpha-\frac{4-2b}{N}\right)-\epsilon\right]M[f]^{-\frac{1-s_c}{s_c}}(y^*)^2\\
&= M[f]^{-\frac{1-s_c}{s_c}}M[Q]^{\frac{1-s_c}{s_c}} 
\underbrace{\left(
\int\left|\sqrt{\mathcal{L}_a} Q\right|^2 + \left(\frac{N-b}{\alpha+2}-\frac{N}{2}\right)\int|x|^{-b}|Q|^{\alpha+2}\right)}_{=0}\\
&\quad\quad\quad\quad\quad-M[f]^{-\frac{1-s_c}{s_c}}\left(\frac{N\alpha+2b}{2}\delta_0 M[Q]^{\frac{1-s_c}{s_c}}E_a[Q]-\epsilon(y^*)^2\right).
    \end{align}
Hence, by taking, say,  $$\eta = M[f]^{-\frac{1-s_c}{s_c}}\delta_0 M[Q]^{\frac{1-s_c}{s_c}}E_a[Q]\frac{N\alpha+2b}{4}>0, $$
the lemma is proved.
\end{proof}

\end{lemma}
An immediate consequence of Lemma \ref{lem_coerc_1} and mass and energy conservation we obtain the following.
\begin{lemma}\label{lem_dicot} Let $N \geq 3$, $a$, $\alpha$ and $b$ as in Theorem \ref{GBInter}, and $u_0 \in H^1(\mathbb{R}^N)$. Denote by $I$ the maximal time of existence of the corresponding solution $u$ to \eqref{INLSa}. If $u_0$ satisfies \eqref{lem_trap_1_concl} and \eqref{lem_trap_below}, then
 \begin{equation}
\sup_{t \in I}\|u_0 \|^{\frac{1-s_c}{s_c}}_{L^2}\|\sqrt{\mathcal{L}_a} u(t) \|_{L^2}  \leq (1-\delta) \|Q\|^{\frac{1-s_c}{s_c}}_{L^2} \|\sqrt{\mathcal{L}_a} Q\|_{L^2}.
 \end{equation}
In particular, $\sup_{t \in I} \|u(t)\|_{H^1_a} < +\infty$, which implies $I = (-\infty, +\infty)$. Alternatively, if $u_0$ satisfies \eqref{lem_trap_1_concl} and \eqref{lem_trap_above}, then for all $t \in I$, 
\begin{equation}
\|u_0 \|^{\frac{1-s_c}{s_c}}_{L^2}\|\sqrt{\mathcal{L}_a} u(t) \|_{L^2}  \geq(1+\delta) \|Q\|^{\frac{1-s_c}{s_c}}_{L^2} \|\sqrt{\mathcal{L}_a} Q\|_{L^2}.
\end{equation}
\end{lemma}
\begin{proof}[\bf {Proof of Theorem \ref{GBInter}}] Theorem \ref{GBInter}a) follows directly from Lemma \ref{lem_dicot}.

Now we prove Theorem \ref{GBInter}b).
Given Lemmas \ref{lem_coerc_2} and \ref{lem_dicot}, the result for $|x|u_0 \in L^2$ follows from Lemma \ref{Morawetz}, which gives, in this context:
\begin{equation}
 \partial^2_{tt}\int |x|^2|u(t)|^2dx =  8\left[\int\left|\sqrt{\mathcal{L}_a} f\right|^2dx + \left(\frac{N-b}{\alpha+2}-\frac{N}{2}\right)\int|x|^{-b}|f|^{\alpha+2}dx\right].
\end{equation}
By \eqref{coerc_blow}, we then have, for all $t$ in the interior of the maximal interval of existence of $u$,
\begin{equation}
\partial^2_{tt}\int |x|^2|u(t)|^2dx \leq -8\eta < 0,
\end{equation}
which shows that $u$ blows up in both finite positive and negative times.

To prove blow-up in the radial case, we employ Lemma \ref{Morawetz} again. Define $\psi: (0,+\infty) \to \mathbb{R}$ as a smooth function satisfying
\begin{equation}
    \psi(r) = \begin{cases}
    r^2, & r \leq 1,\\
    4, &r \geq 2
    \end{cases}
\end{equation}
and we also impose that $\psi''(r) \leq 2$ for all $r >0$. Define also, for $R>0$, $\phi_R(x) = R^2 \psi(x/R)$ and 
\begin{equation}
    V_R(t) = \int \phi_R |u(t)|^2dx.
\end{equation}
By Lemma \ref{Morawetz}, the non-negativity of $\psi''$ and the radiality of $u$,
\begin{align}
    V_R''(t) &= 8\int\left|\sqrt{\mathcal{L}_a} f\right|^2 dx+ 8\left(\frac{N-b}{\alpha+2}-\frac{N}{2}\right)\int|x|^{-b}|f|^{\alpha+2}dx\\
    &-\left(\frac{4}{\alpha+2}-2\right)\int_{R \leq |x| \leq 2R}\Delta \phi_R |x|^{-b}|u|^{\alpha+2}dx  +\frac{4b}{\alpha+2}\int_{R \leq |x| \leq 2R}x\cdot\nabla \phi_R |x|^{-b-2}|u|^{\alpha+2}dx\\
    &- 4\int_{R \leq |x| \leq 2R} \psi''(|x|/R)|\nabla u|^2dx-4a\int_{R \leq |x| \leq 2R}\frac{ x\cdot\nabla \phi_R}{|x|^4}|u|^2 -\int \Delta\Delta \phi_R |u|^2dx\\
    & \leq 8\left[ \int\left|\sqrt{\mathcal{L}_a} f\right|^2 dx+ \left(\frac{N-b}{\alpha+2}-\frac{N}{2}\right)\int|x|^{-b}|f|^{\alpha+2}dx\right] + O(\frac{1}{R^b}\int_{|x|\geq R} |u|^{\alpha+2}) + O(\frac{1}{R^2} \int|u|^2).
\end{align}
By the coercivity lemmas and mass conservation, we are left to control the middle term in the last inequality, if $R>0$ is large enough. We use Strauss and Young inequalities:
\begin{align}
   \frac{1}{R^b}\int_{|x|\geq R} |u|^{\alpha+2}dx &\lesssim \frac{1}{R^b}\|u\|_{L^{\infty}_{\{|x|\geq R\}}}^{\alpha} \int |u|^2dx\lesssim \frac{1}{R^{\frac{\alpha(N-1)+2b}{2}}}(\|\nabla u\|_{L^2}^{\frac{\alpha}{2}}\|u\|^{\frac{\alpha}{2}}_{L^2})\int |u|^2dx\\ &\lesssim \epsilon \|\nabla u\|_{L^2}^2 + C(\epsilon)\left(\frac{1}{R^{\frac{\alpha(N-1)+2b}{2}}}\|u\|_{L^2}^{2+\frac{\alpha}{2}}dx\right)^{\frac{4}{4-\alpha}}.
\end{align}
Thus, by Lemma \ref{lem_coerc_2}, by choosing $R>0$ depending only on the mass of $u$, we have, for all times
\begin{equation}
    V_R''(t) \leq -7\eta<0,
\end{equation}
which implies blowup in finite positive and negative times.
\end{proof}

\ 

\section{Well-posedness theory via Kato's method}\label{Sec4}

In this section we prove the well-posedness results using the Kato method. The proofs follow from a contraction mapping argument based on the Strichartz estimates. In view of the singular factor $|x|^{-b}$ in the nonlinearity, we will divide our analysis in two regions. Indeed, consider a unit ball, $ B =\{x \in \mathbb{R}^N; |x| \leq 1 \}$. A simple computation reveals that
\begin{equation}\label{xb}
\||x|^{-b}\|_{L^\gamma(B)} < +\infty \quad \textnormal{if}\quad \frac{N}{\gamma}-b > 0
\quad
\textnormal{and}\quad \||x|^{-b}\|_{L^\gamma(B^C)} < +\infty \quad \textnormal{if}\quad \frac{N}{\gamma}-b < 0.
\end{equation}

Moreover, if $F(x, z) =|x|^{-b}|z|^\alpha z$, then
\begin{equation}\label{ID}
|F(x, z) - F(x, w)| \lesssim |x|^{-b}(|z|^\alpha + |w|^\alpha)|z -w|.    
\end{equation}

Before stating the lemmas, we define the norm
\begin{equation}\label{strichartnorm}
\|\langle\sqrt{\mathcal{L}_a}\rangle u\|_{S(L^2;J)}=\| u\|_{S(L^2;J)}+\|\sqrt{\mathcal{L}_a} u\|_{S(L^2;J)},
\end{equation}
where $J\subseteq \mathbb{R}$. For this norms, it is worth mentioning the following:

\begin{remark}\label{re_s_adm}
Since we use the equivalence of Sobolev spaces in $\|\sqrt{\mathcal{L}_a}u\|_{S(L^2)}$, for these terms we only employ $S$-admissible pairs which satisfy the conditions of Remark \ref{equivalence}, that is, $1<r<\frac{N}{s}$ if $a>0$ and $\frac{N}{N-\rho}<r<\frac{N}{s+\rho}$ if $a<0$. We do the same with $\|\sqrt{\mathcal{L}_a}u\|_{S'(L^2)}$. 
\end{remark}

\subsection{Local Theory}
First, we establish good estimates for the nonlinearity in the Strichartz spaces.

\begin{lemma}\label{Lemalocal} Let $N\geq 3$, $0< b<\min\{\frac{N}{2},2\}$ and $a>-\frac{(N-2)^2}{4}+\left(\frac{\alpha(N-2)-(2-2b)}{2(\alpha+1)}\right)^2$. If 
$\max\left\{0,\frac{2-2b}{N-2}\right\}<\alpha<\frac{4-2b}{N-2}$, then the following statement holds  
\begin{itemize}
\item [(i)] $\left \||x|^{-b}|u|^\alpha v\right \|_{L^{2}_IL_x^{\frac{2N}{N+2}}}\;\leq \;c \;(T^{\theta_1}+T^{\theta_2})\| \sqrt{\mathcal{L}_a} u\|^{\alpha}_{S(L^2;I)}\|v\|_{S(L^2;I)}$
\item[(ii)] $\left \|\sqrt{\mathcal{L}_a}(|x|^{-b}|u|^\alpha u)\right \|_{L^{2}_IL_x^{\frac{2N}{N+2}}}\;\leq \;c\; (T^{\theta_1}+T^{\theta_2})\| \sqrt{\mathcal{L}_a} u\|^{\alpha+1}_{S(L^2;I)},$
\end{itemize}
where $c,\theta_1,\theta_2 >0$ and $I=[0,T]$.
\begin{proof} We start with (ii). By using the equivalence of Sobolev spaces (Remark \ref{equivalence}) and dividing the estimate in $B$ and $B^C$, we have
\begin{eqnarray}\label{LLii1}
\begin{split}
\left \|\sqrt{\mathcal{L}_a}(|x|^{-b}|u|^\alpha u)\right \|_{L^{2}_IL_x^{\frac{2N}{N+2}}} &\lesssim \left \|\nabla(|x|^{-b}|u|^\alpha u)\right \|_{L^{2}_IL_x^{\frac{2N}{N+2}}} \\
&\leq \left \|\nabla(|x|^{-b}|u|^\alpha u)\right \|_{L^{2}_IL_x^{\frac{2N}{N+2}}(B^C)}+\left \|\nabla(|x|^{-b}|u|^\alpha u)\right \|_{L^{2}_IL_x^{\frac{2N}{N+2}}(B)}. 
\end{split}
\end{eqnarray}
Let $A \subset \mathbb{R}^N$ denotes either $B$ or $B^C$. Applying H\"older's inequality first in space and then in time, we deduce 
\begin{equation}\label{LLii2}
\begin{split}
 \left\|\nabla(|x|^{-b}|u|^\alpha u)\right\|_{L^2_IL_x^{\frac{2N}{N+2}}(A)} &  \leq \left \|  \||x|^{-b}\|_{L^\gamma(A)} \|\nabla(|u|^\alpha u) \|_{L^{\beta}_x} + \|\nabla(|x|^{-b})\|_{L^d(A)}\|u\|^{\alpha +1}_{L^{(\alpha+1)e}_x}  \right\|_{L^2_I}  \\
&\lesssim  \left\| \| |x|^{-b} \|_{L^\gamma(A)}  \|u\|^\alpha_{L_x^{\alpha r_1}}   \| \nabla u \|_{L_x^{r}} + \||x|^{-b-1}\|_{L^d(A)}\| \nabla u\|^{\alpha +1}_{L^r_x}\right\|_{L^2_I}   \\
&\lesssim T^{\frac{1}{q^*}}\left( \| |x|^{-b} \|_{L^\gamma(A)} \| \nabla u \|^{\alpha+1}_{L^q_IL_x^{r}}+ \||x|^{-b-1}\|_{L^d(A)}\| \nabla u\|^{\alpha +1}_{L^q_IL^r_x}\right),
\end{split}
\end{equation}
where we also have used the Sobolev inequality. Here, we must have the  relations
\[
\begin{cases}
\frac{N+2}{2N}=\frac{1}{\gamma}+\frac{1}{\beta}=\frac{1}{d}+\frac{1}{e},\qquad \frac{1}{\beta}=\frac{1}{r_1}+\frac{1}{r} \\
1=\frac{N}{r}-\frac{N}{\alpha r_1}=\frac{N}{r}-\frac{N}{(\alpha+1)e}\;,\qquad r<N\\
\frac{1}{2}=\frac{1}{q^*}+\frac{\alpha+1}{q},
\end{cases}
\]
which in turn are equivalent to 
\begin{equation}\label{L1C22}
\begin{cases} 
\frac{N}{\gamma}=\frac{N+2}{2}-\frac{N(\alpha+1)}{r}+\alpha\\ 
\frac{N}{d}=\frac{N+2}{2}-\frac{N(\alpha+1)}{r}+\alpha +1\\
\frac{1}{q^*}=\frac{1}{2}-\frac{\alpha+1}{q}.
\end{cases}
\end{equation}
Our goal is to find a pair $(q,r)$ $L^2$-admissible such that $\||x|^{-b}\|_{L^\gamma(B)}$ and $\||x|^{-b-1}\|_{L^d(B)}$ are finite (see \eqref{xb}), $r < N$ and $\frac{1}{q^*}>0$. Let $(q_{\pm},r_{\pm})$ defined by\footnote{It is easy to see that $q\geq 2$. Moreover, note that the denominator of $q$ is positive for $\alpha>0$, if $b\geq 1$ and $\alpha>\frac{2-2b}{N-2}$, if $b<1$.}
$$
r_{\pm}=\frac{2N(\alpha+1)}{N+2+2\alpha-2b\pm \varepsilon} \quad \textnormal{and} \quad q_{\pm}=\frac{4(\alpha+1)}{\alpha(N-2)-(2-2b) \mp \varepsilon},
$$
for $\varepsilon>0$ sufficiently small. Indeed, we can easily see that $(q_\pm,r\pm)$ is $S$-admissible, $r_\pm<N$ (here we need to use that $b<\frac{N}{2}$) and
$$
\frac{1}{q^*_\pm}=\frac{1}{2}-\frac{\alpha+1}{q_\pm}=\frac{4-2b-\alpha(N-2)\pm\varepsilon}{4}.
$$
Now, if $A=B^C$ we choose $(q,r)=(q_+,r_+)$ and  $\theta_1=\frac{1}{q^*_+}$. Then, $\frac{N}{\gamma}-b>0$ and $\frac{N}{d}-b-1>0$, and
consequently, $\||x|^{-b}\|_{L^\gamma (B^C)},\||x|^{-b-1}\|_{L^d(B^C)}<\infty$, by \eqref{xb}. On the other hand, if $A=B$ we choose the pair
$(q,r)= (q_-,r_-)$ and $\theta_2=\frac{1}{q^*_-}$, so\footnote{Since, $\alpha<\frac{4-2b}{N-2}$ we have $\theta_1,\theta_2>0$.} we also get $\||x|^{-b}\|_{L^\gamma (B)},\||x|^{-b-1}\|_{L^d(B)}<\infty$. Hence, the relations \eqref{LLii1} and \eqref{LLii2} implies that
$$
\left \|\sqrt{\mathcal{L}_a}(|x|^{-b}|u|^\alpha u)\right \|_{L^{2}_IL_x^{\frac{2N}{N+2}}} \lesssim (T^{\theta_1}+T^{\theta_2})\| \nabla u\|^{\alpha+1}_{L^q_IL^r_x}.
$$
Finally, if $\frac{N}{N-\rho}<r<\frac{N}{1+\rho}$ and applying again Remark \ref{equivalence} we complete the proof of part (ii). Indeed, using the value of $\rho$ given in \eqref{ro}, it is easy
to see that $r>\frac{N}{N-\rho}$. In addition, $r<\frac{N}{1+\rho}$ is equivalent to $2\rho (\alpha+1)<N-2b$ with $b<\frac{N}{2}$, which is true assuming our hypothesis on $a$. 

\ The proof of (i) is essentially the same as in (ii). It means, we have
\begin{eqnarray*}
\left \||x|^{-b}|u|^\alpha v\right \|_{L^{2}_IL_x^{\frac{2N}{N+2}}(A)} &\lesssim &\left\| \| |x|^{-b} \|_{L^\gamma(A)}  \|u\|^\alpha_{L_x^{\alpha r_1}}   \|v \|_{L_x^{r}}\right\|_{L^2_I}\\
&\lesssim & T^{\frac{1}{q^*}} \| |x|^{-b} \|_{L^\gamma(A)} \| \nabla u \|^{\alpha}_{L^q_IL_x^{r}} \| v \|_{L^q_IL_x^{r}},
\end{eqnarray*}
where $\frac{N}{\gamma}=\frac{N+2}{2}-\frac{N(\alpha+1)}{r}+\alpha$ and $\frac{1}{2}=\frac{1}{q^*}+\frac{\alpha+1}{q}$. Choosing $(q,r)$ and arguing exactly as in (ii) we obtain (i). 
\end{proof}
\end{lemma}

\ In the next lemma we consider the case, $\alpha\leq \frac{2-2b}{N-2}$ for $0<b<1$.

\begin{lemma}\label{Lemalocal1} Let $N\geq 3$, $0< b<1$ and $a>-\frac{(N-2)^2}{4}$. If $\frac{2-2b}{N}<\alpha\leq\frac{2-2b}{N-2}$, then the following statement hold
\begin{itemize}
\item [(i)] $\left \||x|^{-b}|u|^\alpha v\right \|_{L^{2}_IL_x^{\frac{2N}{N+2}}}\;\leq\; c \left(T^{\theta_1}+T^{\theta_2}\right)\|\langle \sqrt{\mathcal{L}_a}\rangle u\|^{\alpha}_{S(L^2;I)}\|v\|_{S(L^2;I)}$
\item[(ii)] $\left \|\sqrt{\mathcal{L}_a}(|x|^{-b}|u|^\alpha u)\right \|_{L^{2}_IL_x^{\frac{2N}{N+2}}}\;\leq \;c \left(T^{\theta_1}+T^{\theta_2}\right)\| \langle\sqrt{\mathcal{L}_a}\rangle u\|^{\alpha+1}_{S(L^2;I)},$
\end{itemize}
where $c,\theta_1,\theta_2 >0$ and $I=[0,T]$.
\begin{proof} 
We first estimate (ii). We divide in two regions, $B$ and $B^C$. Similarly as the previous lemma, it follows that
\begin{equation}
\begin{split}
 \left\|\nabla(|x|^{-b}|u|^\alpha u)\right\|_{L^2_IL_x^{\frac{2N}{N+2}}(B)} 
&\lesssim  \left\| \| |x|^{-b} \|_{L^\gamma(B)}  \|u\|^\alpha_{L_x^{\alpha r_1}}   \| \nabla u \|_{L_x^{2}} + \||x|^{-b-1}\|_{L^d(B)}\| u\|^{\alpha }_{L^{\alpha r_1}_x} \|u\|_{L^{\frac{2N}{N-2}}_x}\right\|_{L^2_I}   \\
&\lesssim T^{\frac{1}{q^*}}\left( \| \nabla u \|^{\alpha}_{L^q_IL_x^{r}}+ \| \nabla u\|^{\alpha }_{L^q_IL^r_x}\right)\| \nabla u \|_{L_x^{2}},
\end{split}
\end{equation}
where 
\[
\begin{cases}
\frac{N+2}{2N}=\frac{1}{\gamma}+\frac{1}{r_1}+\frac{1}{r}=\frac{1}{d}+\frac{1}{r_1}+\frac{N-2}{2N} \\
1=\frac{N}{r}-\frac{N}{\alpha r_1}\\
\frac{1}{2}=\frac{1}{q^*}+\frac{\alpha}{q},
\end{cases}
\]
which implies that
\begin{equation}
\begin{cases} 
\frac{N}{\gamma}=\frac{N}{d}-1=1-\frac{N\alpha}{r}+\alpha\\
\frac{1}{q^*}=\frac{1}{2}-\frac{\alpha+1}{q}.
\end{cases}
\end{equation}
For $\varepsilon > 0$ small, by choosing the $S$-admissible pair $(q, r)$ defined by\footnote{Note that, the denominator of $r$ is positive since $b<1$.}
$$
q=\frac{4(2-2b+\varepsilon)}{\varepsilon(N-2)}\quad\textnormal{and}\quad r=\frac{N(2-2b+\varepsilon)}{(1-b)N+\varepsilon},
$$
we deduce  that $\frac{N}{\gamma}-b>0$ and $\frac{N}{d}-1-b>0$ (assuming, $\alpha\leq \frac{2-2b}{N-2}$). It leads to $\||x|^{-b} \|_{L^\gamma(B)} $ and $\||x|^{-b-1} \|_{L^d(B)}$ are finite. Moreover, it is easy to see that $\frac{1}{q*}>0$. On the other hand, since $\frac{N}{N-\rho}<r<\frac{N}{1+\rho}$ and\footnote{Using the value of $\rho$, it is easy to check $r>\frac{N}{N-\rho}$. Moreover, chossing $\varepsilon<\frac{(1-b)\sqrt{(N-2)^2+4a}}{\rho}$ we have $r<\frac{N}{1+\rho}$.} by the equivalence of Sobolev spaces one has
$$
\left\|\sqrt{\mathcal{L}_a}(|x|^{-b}|u|^\alpha u)\right\|_{L^2_IL_x^{\frac{2N}{N+2}}(B)} 
\lesssim T^{\theta_1} \| \sqrt{\mathcal{L}_a}u\|^\alpha_{L^q_IL^r_x} \|u\|_{L^\infty_IH^1_a},
$$
where $\theta_1=\frac{1}{q^*}$. We now consider the estimate on $B^C$. The H\"older inequality, equivalence of Sobolev spaces and Sobolev embedding imply that 
\begin{align*}
 \left\|\nabla(|x|^{-b}|u|^\alpha u)\right\|_{L^2_IL_x^{\frac{2N}{N+2}}(B^C)} 
&\lesssim  \left\| \| |x|^{-b} \|_{L^\gamma(B^C)}  \|u\|^\alpha_{L_x^{\alpha r_1}}   \| \nabla u \|_{L_x^{2}} + \||x|^{-b-1}\|_{L^d(B^C)}\| u\|^{\alpha }_{L^{\alpha r_1}_x} \|u\|_{L^{\frac{2N}{N-2}}_x}\right\|_{L^2_I}   \\
&\lesssim T^{\frac{1}{2}}\| u\|^{\alpha }_{L^\infty_IH^1_x}\| \nabla u \|_{L^\infty_IL_x^{2}}
\lesssim T^{\theta_2}\| u\|^{\alpha +1}_{L^\infty_IH^1_a},
\end{align*}
where $\theta_2=\frac{1}{2}$ and
$$
\frac{N}{\gamma}=\frac{N}{d}-1=1-\frac{N}{r_1}.
$$
We need to show that $H^1\hookrightarrow L^{\alpha r_1}$ and $\frac{N}{\gamma}-b=\frac{N}{d}-b-1<0$. To this end, we choose $r_1=\frac{N}{1-b+\varepsilon}$ for $\varepsilon>0$ small, we have $2< \alpha r_1< \frac{2N}{N-2}$ by hypothesis\footnote{Note that in the particular case, $b=0$, if $\alpha\geq  \frac{2}{N}$ then $\alpha r_1\geq 2$, so $H^1\hookrightarrow L^{\alpha r_1}$. That is, in this case we can consider $\alpha=\frac{2}{N}$. }  $\frac{2(1-b)}{N}<\alpha \leq \frac{2(1-b)}{N-2}$.

\ To show (i) is only replace $|x|^{-b}|u|^\alpha \nabla u$ by $|x|^{-b}|u|^\alpha v$ in the proof of (ii). This completes the proof of Lemma \ref{Lemalocal1}.
\end{proof}
\end{lemma}

Now, with the previous lemmas in hand we are in a position to prove Theorem \ref{LWP}.
\begin{proof}[\bf{Proof of Theorem \ref{LWP}}]	
We use the contraction mapping principle. To do so, we define $$
X= C\left([0,T];H_a^1(\mathbb{R}^N)\right)\bigcap L^q\left([0,T];H_a^{1,r}(\mathbb{R}^N)\right),$$ where $(q,r)$ is any $S$-admissible pair and $T>0$ will be determined properly later. 
We shall show that 
\begin{equation}\label{OPERATOR}
G(u)(t)=e^{-it\mathcal{L}_a}u_0+i\lambda \int_0^t e^{-i(t-t')\mathcal{L}_a}|x|^{-b}|u(t')|^\alpha u(t')dt'
\end{equation}
is a contraction on the complete metric space $S(m,T)=\{u \in X : \|\langle \sqrt{\mathcal{L}_a} \rangle  u \|_{S^2(L^2;I)}\leq m \}
$
with the metric 
$$
d_T(u,v)=\|u-v\|_{S\left(L^2;I\right)},
$$
where $I=[0.T]$ and $\|\langle \sqrt{\mathcal{L}_a} \rangle  u \|_{S^2(L^2;I)}$ is defined in \eqref{strichartnorm}.

\ Let us first show that $G$ is well defined from $S(m,T)$ to $S(m,T)$. Indeed, it follows from the Strichartz inequalities in Lemma \ref{Strichartzpotential0} together with Lemmas \ref{Lemalocal} and \ref{Lemalocal1} that
\begin{equation}\label{NCD0} 
\begin{split} 
\| \langle \sqrt{\mathcal{L}_a} \rangle G(u)\|_{S\left(L^2;I\right)}&\leq c\|\langle \sqrt{\mathcal{L}_a} \rangle u_0\|_{L^2}+ \left \| \langle \sqrt{\mathcal{L}_a} \rangle (|x|^{-b}|u|^\alpha u)\right \|_{L^{2}_IL_x^{\frac{2N}{N+2}}}\\
&\leq  c\|u_0\|_{H^1_a}+ c \left(T^{\theta_1}+T^{\theta_2}\right) \| \langle \sqrt{\mathcal{L}_a} \rangle u \|^{\alpha+1}_{S\left(L^2;I\right)},
\end{split}
\end{equation}
Consequently, by choosing $m=2c\|u_0\|_{H^1_a}$ and $T>0$ such that 
\begin{equation}\label{CTHs} 
c m^{\alpha}(T^{\theta_1}+T^{\theta_2} )< \frac{1}{4},
\end{equation}
we obtain $G(u)\in S(m,T)$. Hence, $G$ is well defined on $S(m,T)$.

\ To prove that $G$ is a contraction on $S(m,T)$ with respect to the metric $d_T$. As before and by \eqref{ID} we deduce 
\[
\begin{split}
d_T(G(u),G(v))&\leq   c\left\| |x|^{-b}(|u|^\alpha u-|v|^\alpha v \right)\|_{L^{2}_IL_x^{\frac{2N}{N+2}}}
\\ 
&\leq  c T^{\theta}\left(\|\sqrt{\mathcal{L}_a}u\|^\alpha_{S\left(L^2;I\right)}+\|\sqrt{\mathcal{L}_a}v\|^\alpha_{S\left(L^2;I\right)}\right)\|u-v\|_{S\left(L^2;I\right)}.
\end{split}
\]
So, if $u,v\in S(m,T)$, then $
 d_T(G(u),G(v))\leq c (T^{\theta_1}+T^{\theta_2})a^\alpha d_T(u,v).
$
Therefore, from \eqref{CTHs}, $G$ is a contraction on $S(m,T)$ and by the contraction mapping principle we have a unique fixed point $u \in S(m,T)$ of $G$. This completes the proof of the theorem.
\end{proof}

\subsection{\bf Small Global Theory}\label{secglo}

\ In this subsection, we turn our attention to prove the small data global results. Similarly as in the local theory, we establish suitable estimates on the nonlinearity\footnote{When $u=v$, we denote $F(x,u,v)$ by $F(x,u)$.} $F(x,u,v)=|x|^{-b}|u|^\alpha v$. It is worth mentioning that, since \eqref{SE5} holds for radial data, we obtain global results for radial initial data and also nonradial data. To this end, we use the norms $\|u\|_{S(\dot{H}^{s_c})}$ and $\|u\|_{\tilde{S}(\dot{H}^{s_c})}$, respectively.

\ We first obatin estimates for $a>0$ and in the sequel for $a<0$. For $a>0$, we use the results obtained by the second author \cite{CARLOS}. Recalling the numbers used in \cite{CARLOS}.
\begin{equation}\label{PHsA1}  
\widehat{q}=\frac{4\alpha(\alpha+2-\theta)}{\alpha(N\alpha+2b)-\theta(N\alpha-        4+2b)},\;\;\;\widehat{r}\;=\;\frac{N\alpha(\alpha+2-\theta)}{\alpha(N-b)-\theta(2-b)},
\end{equation}
and
\begin{equation}\label{PHsA2}
\widetilde{a}\;=\;\frac{2\alpha(\alpha+2-\theta)}{\alpha[N(\alpha+1-\theta)-2+2b]-(4-2b)(1-\theta)},\;\;\;  \widehat{a}=\frac{2\alpha(\alpha+2-\theta)}{4-2b-(N-2)\alpha}.
\end{equation}
It is easy to see that $(\widehat{q},\widehat{r})$ is $L^2$-admissible, $(\widehat{a},\widehat{r})$ is $\dot{H}^{s_c}$-admissible and $(\widetilde{a},\widehat{r})$ is $\dot{H}^{-s_c}$-admissible. 

\ The first lemma will be used to prove the global well posedness in the radial case.
\begin{lemma}\label{LG1}
Let $N\geq 3$, $\frac{4-2b}{N}<\alpha<\frac{4-2b}{N-2}$ and $0<b<\min\{\frac{N}{2},2\}$. If $a> 0$, then there exists $\theta\in (0,\alpha)$ sufficiently small such that 
\begin{itemize}
\item [(i)] $\left \|F(x,u,v) \right\|_{S'(\dot{H}^{-s_c})} \;\lesssim \; 
\| u\|^{\theta}_{L^\infty_tH^1_a}\|u\|^{\alpha-\theta}_{S(\dot{H}^{s_c})} \|v\|_{S(\dot{H}^{s_c})}$
\item [(ii)] $\left\|F(x,u,v)\right\|_{S'(L^2)}
\;\lesssim \; 
\| u\|^{\theta}_{L^\infty_tH^1_a}\|u\|^{\alpha-\theta}_{S(\dot{H}^{s_c})} \| v\|_{S(L^2)}
$
\item [(iii)] $\left\|\sqrt{\mathcal{L}_a} F(x,u)\right\|_{L^{\widehat{q}'}_tL^{\widehat{r}'}_x}
\;\lesssim\; \|u\|^{\theta}_{L^\infty_tH^1_a}\|u\|^{\alpha-\theta}_{S(\dot{H}^{s_c})} \| \sqrt{\mathcal{L}_a}  u\|_{S(L^2)}
$\;  if \;$N\geq 4$ 
\item [(iv)] $\left\|\sqrt{\mathcal{L}_a} F(x,u)\right\|_{L^{2}_tL_x^{\frac{6}{5}}}
\;\lesssim \; \| u\|^{\theta}_{L^\infty_tH^1_a}\|u\|^{\alpha-\theta}_{S(\dot{H}^{s_c})} \|\sqrt{\mathcal{L}_a}  u\|_{S(L^2)}$ \; if \;$N=3$\;\;and\;\;$\alpha<3-2b$.
\end{itemize}
\begin{proof} 
See \cite[Lemmas $4.1$ and $4.2$, with $s=1$]{CARLOS} to show (i) and (ii), respectively\footnote{To show (i), the pair used was ($\widetilde{a},\widehat{r})$ $\dot{H}^{-s_c}$-admissible.}. To show (iii) we used the estimate used in \cite[Lemma 4.3]{CARLOS} with $s=1$, i.e.,
$$
\left\|\nabla F(x,u)\right\|_{L^{\widehat{q}'}_tL^{\widehat{r}'}_x}
\lesssim \|u\|^{\theta}_{L^\infty_tH^1_a}\|u\|^{\alpha-\theta}_{L^{\widehat{a}}_tL^{\widehat{r}}_x} \| \nabla  u\|_{L^{\widehat{q}}_tL^{\widehat{r}}_x}.
$$
We notice that, if $N\geq 4$ then $\widehat{r}<N$ and $\widehat{r}'<N$, so the equivalence of Sobolev spaces (Remark \ref{equivalence} with $s=1$) implies (iii). Finally, we consider (iv). To this end, we use the following numbers\footnote{We use other admissible pairs since $\widehat{r}<N$ is not true for $N=3 $.}.
\begin{align}\label{pairN=3a}
q_\varepsilon\;=\;\frac{4}{1-2\varepsilon}\;,\;\;\;\;r_\varepsilon=\frac{3}{1+\varepsilon}
\end{align}
and
\begin{align}\label{pairN=3b}
a\;=\;\frac{8(\alpha-\theta)}{1+2\varepsilon}\;,\;\;\;\;\;\;\;r=\frac{6\alpha(\alpha-\theta)}{\alpha(3-2b-2\varepsilon)-2\theta (2-b)},
\end{align}
where $\varepsilon$ is small. Observe that $r_\varepsilon<3$ and the denominator of $r$ is positive if $b<\frac{3}{2}$. Moreover, an easy computation shows that $(a,r)$ is $\dot{H}^{s_c}$-admissible if\footnote{The condition $\alpha<3-2b$ implies that $r<6$, condition of $S(\dot{H}^{s_c})$-admissible pair, see \eqref{HsAdmissivel}.} $\alpha<3-2b$ and $(q_\varepsilon,r_\varepsilon)$ is $S$-admissible.

Let $E(t)=\left\|\nabla(|x|^{-b}|u|^\alpha u)\right\|_{L_x^{\frac{6}{5}}(A)}$, where $A$ denotes either $B$ or $B^C$. The H\"older inequality and the Sobolev embedding lead to
\[
\begin{split}\label{LG2R0}
E(t )&\leq  \||x|^{-b}\|_{L^\gamma(A)}\|u\|^\theta_{L_x^{\theta r_1}}  \|u\|^{\alpha-\theta}_{L_x^r}   \| \nabla u \|_{L_x^{r_\varepsilon}}  + \||x|^{-b-1}\|_{L^d(A)}\|u\|^\theta_{L_x^{\theta r_1}}\|u\|^{\alpha-\theta}_{L^r_x} \|u\|_{L_x^{e}}   \\
&\lesssim   \|u\|^\theta_{L_x^{\theta r_1}}  \|u\|^{\alpha-\theta}_{L_x^r} \| \nabla u \|_{L_x^{r_\varepsilon}} + \|u\|^\theta_{L_x^{\theta r_1}}\|u\|^{\alpha-\theta}_{L^r_x} \|\nabla u\|_{L_x^{r_\varepsilon}},
\end{split}
\]
where (using $1=\frac{3}{r_\varepsilon}-\frac{3}{e}$) 
\begin{equation} 
\begin{cases}
\frac{3}{\gamma}=\frac{3}{2}+1-\frac{3}{r_1}-\frac{3(\alpha-\theta)}{r}-\frac{3}{r_\varepsilon}\\ 
\frac{3}{d}=\frac{3}{2}+1-\frac{3}{r_1}-\frac{3(\alpha-\theta)}{r}-\frac{3}{e}=\frac{3}{2}+1-\frac{3}{r_1}-\frac{3(\alpha-\theta)}{r}-(\frac{3}{r_\varepsilon}-1),
\end{cases}
\end{equation}
which implies using the definition of the numbers $r$ and $r_\varepsilon$ (see \eqref{pairN=3a}-\eqref{pairN=3b}) that
\begin{equation}\label{RI}
\frac{3}{\gamma}-b=\frac{3}{d}-b-1=\frac{\theta(2-b)}{\alpha}-\frac{3}{r_1}.    
\end{equation}
If $A=B$ we  choose $\theta r_1=\frac{2N}{N-2}$, so that $\frac{N}{\gamma}-b=\theta(1-s_c)>0$ (recall that $s_c<1$). On the other hand, if $A=B^C$ we choose $\theta r_1=2$, so that $\frac{N}{\gamma}-b=-\theta s_c<0$. Thus, the quantities $\||x|^{-b}\|_{L^\gamma(A)}$, $\||x|^{-b-1}\|_{L^d(A)}<\infty$ and $H^1_x\hookrightarrow L^{\theta r_1}$. Therefore, since $\frac{1}{2}=\frac{\alpha-\theta}{a}+\frac{1}{q_\varepsilon}$ one has 
$$
\left\|\nabla F(x,u)\right\|_{L^{2}_tL^{\frac{6}{5}}_x}
\lesssim  \|u\|^{\theta}_{L^\infty_tH^1_x}\|u\|^{\alpha-\theta}_{L^{a}_tL^{r}_x} \| \nabla  u\|_{L^{q_\varepsilon}_tL^{r_\varepsilon}_x}.
$$
Hence, applying the equivalence of Sobolev spaces (since $r_\epsilon<3$) we conclude with the proof of (iv).   
\end{proof}
\end{lemma}

\ Our goal now is to show the radial small data result (Theorem \ref{GWP1}). 

\begin{proof}[\bf{Proof of Theorem \ref{GWP1}}]
As usual, our proof is based on the contraction mapping principle. Indeed, define
$$
S=\{ u:\;\|u\|_{S(\dot{H}^{s_c})}\leq 2\|e^{-it\mathcal{L}_a}u_0\|_{S(\dot{H}^{s_c})}\quad\textnormal{and}\quad\|\langle \sqrt{\mathcal{L}_a}\rangle u\|_{S(L^2)}\leq 2c\|u_0\|_{H^1_a}\}.
$$
We shall show that $G=G_{u_0}$ defined in \eqref{OPERATOR} is a contraction on $S$ equipped with the metric 
$$
d(u,v)=\|u-v\|_{S(L^2)}+\|u-v\|_{S(\dot{H}^{s_c})}.
$$

\ Indeed, we deduce by the Strichartz inequalities (\ref{SE1}), (\ref{SE2}), \eqref{SE3} and \eqref{SE5}
\begin{equation}
\|G(u)\|_{S(\dot{H}^{s_c})}\leq \|e^{-it\mathcal{L}_a}u_0\|_{S(\dot{H}^{s_c})}+ c\| F(x,u) \|_{S'(\dot{H}^{-s_c})}
\end{equation}
\begin{equation}
\|G(u)\|_{S(L^2)}\leq c\|u_0\|_{L^2}+ c\| F(x,u)\|_{S'(L^2)}
\end{equation}
and	
\begin{equation}
\|\sqrt{\mathcal{L}_a} G(u)\|_{S(L^2)}\leq c \|\sqrt{\mathcal{L}_a} u_0\|_{L^2}+ c\|\sqrt{\mathcal{L}_a} F(x,u)\|_{S'(L^2)},
\end{equation}
where $F(x,u)=|x|^{-b}|u|^\alpha u$. On the other hand, it follows from Lemma \ref{LG1} and the three last inequalities that, for $u\in S$
\begin{align*}
\|G(u)\|_{S(\dot{H}^{s_c})}\leq& \|e^{-it\mathcal{L}_a}  u_0\|_{S(\dot{H}^{s_c})} +c\| u \|^\theta_{L^\infty_tH^1_a}\| u \|^{\alpha-\theta}_{S(\dot{H}^{s_c})}\| u \|_{S(\dot{H}^{s_c})} \nonumber \\
\leq & \|e^{-it\mathcal{L}_a} u_0\|_{S(\dot{H}^{s_c})}+2^{\alpha+1}c^{\theta+1} M^\theta \| e^{-it\mathcal{L}_a} u_0 \|^{\alpha-\theta+1}_{S(\dot{H}^{s_c})}
\end{align*}
and
\begin{equation}
\begin{split}
\|\langle \sqrt{\mathcal{L}_a}\rangle G(u)\|_{S(L^2)} &\leq  c\|u_0\|_{H^1_a}+c\|u\|^\theta_{L^\infty_tH^1_a}\| u \|^{\alpha-\theta}_{S(\dot{H}^{s_c})} \|\langle \sqrt{\mathcal{L}_a}\rangle u\|_{S(L^2)}\\
& \leq  c\|u_0\|_{H^1_a}+c^{\theta+2}2^{\alpha+1} M^{\theta}\|e^{-it\mathcal{L}_a}  u_0 \|^{\alpha-\theta}_{S(\dot{H}^{s_c})}\|u_0\|_{H^1_a}.
\end{split}
\end{equation}

\ Now if $\| U(t)u_0 \|_{S(\dot{H}^{s_c})}<\delta$ with 
\begin{equation}\label{WD}
\delta\leq \min\sqrt[\alpha-\theta]{\frac{1}{2c^{\theta+1}2^{\alpha+1}M^\theta}},
\end{equation}
where $A>0$ is a number such that $\|u_0\|_{H^s}\leq A$, then
$$\|G(u)\|_{\widetilde{S}(\dot{H}^{s_c})}\leq 2\| e^{-it\mathcal{L}_a}u_0 \|_{\widetilde{S}(\dot{H}^{s_c})}\quad \mbox{and}\quad \|\langle \sqrt{\mathcal{L}_a}\rangle G(u) \|_{S(L^2)}\leq 2c\|u_0\|_{H^1_a},$$
that is  $G(u)\in S$.

\ To prove that $G$ is a contraction on $S$, we repeat the above computations. Indeed, taking $u,v\in S$ 
\begin{equation}
\|G(u)-G(v)\|_{S(\dot{H}^{s_c})}\leq 2^{\alpha+1}c^{\theta+1} M^\theta\|e^{-it\mathcal{L}_a}u_0 \|^{\alpha-\theta}_{B(\dot{H}^{s_c})} \|u-v\|_{S(\dot{H}^{s_c})}
\end{equation}
and
\begin{equation}
\|G(u)-G(v)\|_{S(L^2)}\leq 2^{\alpha+1}c^{\theta+1} M^\theta\|e^{-it\mathcal{L}_a}u_0 \|^{\alpha-\theta}_{B(\dot{H}^{s_c})} \|u-v\|_{S(L^2)}.
\end{equation}
By using the last inequalities and \eqref{WD} we obtain
$$
d(G(u),G(v)) \leq 2^{\alpha+1}c^{\theta+1} M^\theta\|e^{-it\mathcal{L}_a}u_0 \|^{\alpha-\theta}_{B(\dot{H}^{s_c})} \;d(u,v)\leq \frac{1}{2}d(u,v),
$$
i.e., $G$ is a contraction.

\ Therefore, by the Banach Fixed Point Theorem, $G$ has a unique fixed point $u\in S$, which is a global solution of \eqref{INLSa}.
\end{proof}
\begin{remark}
Observe that Theorem \ref{LWP} (when $N=3$) shows global well posedness with an extra condition on $\alpha$, i.e., $\alpha<3-2b$. We now establish suitable estimates on $F(x,u,v)$ for the full intercritical regime, that is, $\frac{4-2b}{3}<\alpha<4-2b$.
\end{remark}

\ Before stating the lemma, we first define the following numbers
\begin{align}\label{LG2N=3a}
\bar{a}=
\frac{4(\alpha-\theta)}{1+2\varepsilon}
\qquad \bar{r}=
\frac{6\alpha(\alpha-\theta)}{\alpha(3-2b-2\varepsilon)-\theta (4-2b)}
\end{align}
and
\begin{align}\label{LG2N=2b}
q=\frac{4}{1-2\varepsilon}\qquad r=\frac{3}{1+\varepsilon}.
\end{align}
Observe that $(q,r$) is $S$-admissible and $r<N$. Moreover, $\frac{1}{2}=\frac{\alpha-\theta}{\bar{a}}+\frac{1}{q}$.
\begin{lemma}\label{LG2} Let $N=3$ and $a> 0$. If $0<b<\frac{3}{2}$ and $\max\{\frac{4-2b}{3},1\}< \alpha<4-2b$, then the following statements hold  
\begin{itemize}
\item [(i)] $
\left\| F(x,u,v)\right\|_{L^{2}_tL_x^{\frac{6}{5}}}\;\lesssim\;
\| u\|^{\theta}_{L^\infty_tH^1_a}\|u\|^{\alpha-\theta}_{L^{\bar{a}}_tL^{\bar{r}}_x} \|v\|_{S(L^2)}$
\item [(ii)] $
\left\|\sqrt{\mathcal{L}_a} F(x,u)\right\|_{L^{2}_tL_x^{\frac{6}{5}}}\;\lesssim\;
\| u\|^{\theta}_{L^\infty_tH^1_a}\|u\|^{\alpha-\theta}_{L^{\bar{a}}_tL^{\bar{r}}_x} \|\sqrt{\mathcal{L}_a}   u\|_{S(L^2)}$ 
\item [(iii)] $
\left\|\sqrt{\mathcal{L}_a}^{s_c} F(x,u)\right\|_{L^{2}_tL_x^{\frac{6}{5}}} \;\lesssim\;
\| u\|^{\theta}_{L^\infty_tH^1_a}\|u\|^{\alpha-\theta}_{L^{\bar{a}}_tL^{\bar{r}}_x} \|\langle\sqrt{\mathcal{L}_a}  \rangle u\|_{S(L^2)}$, 
\end{itemize}
where $\theta$ is sufficiently small.
\begin{proof}
Let us prove (i). Similarly as in Lemma \ref{LG1} (iv), it follows that (since $\frac{1}{2}=\frac{\alpha-\theta}{\bar{a}}+\frac{1}{q}$)
\[
\begin{split}
\left\|F(x,u,v)\right\|_{L^2_tL_x^{\frac{6}{5}}(A)}&\leq  \||x|^{-b}\|_{L^\gamma(A)}\|u\|^\theta_{L^\infty_tL_x^{\theta r_1}}  \|u\|^{\alpha-\theta}_{L_t^{\bar{a}}L_x^{\bar{r}}}   \| v \|_{L^{q}_tL_x^{r}},
\end{split}
\]
where $A$ denotes either $B$ or $B^C$. Moreover $\frac{3}{\gamma}=\frac{5}{2}-\frac{3}{r_1}-\frac{3(\alpha-\theta)}{\bar{r}}-\frac{3}{r}$. From \eqref{LG2N=3a}, \eqref{LG2N=2b} we have
$$
\frac{3}{\gamma}-b=\frac{\theta(2-b)}{\alpha}-\frac{3}{r_1},
$$
which is the same relation as in \eqref{RI}, so choosing $r_1$ as in Lemma \ref{LG1} (iv) we deduce that $|x|^{-b}\in L^\gamma(A)$ and $H^1_x\hookrightarrow L^{\theta r_1}$. Hence, in view of 
($q,r$) is $S$-admissible we obatin (i). We now consider (ii). The equivalence of Sobolev spaces implies that
\begin{equation}\label{LG2ii1}
\left\|\sqrt{\mathcal{L}_a} (|x|^{-b}|u|^\alpha u)\right\|_{L_t^{2}L^{\frac{6}{5}}_x}\leq \left\|\nabla (|x|^{-b}|u|^\alpha u)\right\|_{L_t^{2}L^{\frac{6}{5}}_x}.    
\end{equation}
The derivate of $F$ leads to two terms, one of the form $|x|
^{-b}|u|^\alpha
\nabla u$ and one of the form $|x|
^{-b-1}|u|^\alpha u$, so for  estimating the  first  one we replace $v$ by $\nabla u$ in (i). Thus, 
\begin{eqnarray*}
\left\| |x|^{-b}|u|^\alpha \nabla u\right\|_{L^{2}_tL^{\frac{6}{5}}_x}
&\lesssim & \|u\|^{\theta}_{L^\infty_tH^1_a}\|u\|^{\alpha-\theta}_{L^{\bar{a}}_tL^{\bar{r}}_x} \|  \nabla u\|_{L^{q}_tL^{r}_x}.
\end{eqnarray*}
Furthermore, as in Lemma \ref{LG1} (iv) we get 
\[
\begin{split}
\left\||x|^{-b-1}|u|^\alpha u\right\|_{L_t^{2}L^{\frac{6}{5}}_x(A)} &\leq  \left\|\||x|^{-b-1}\|_{L^d(A)}\|u\|_{L_{x}^{\theta r_1}}^{\theta} \|u\|^{\alpha-\theta}_{L_{x}^{(\alpha-\theta)r_2}}\|u\|_{L^{r_3}_x}\right\|_{L^{2}_t}\\
& \lesssim   \|u\|_{L^\infty_t H_{x}^{1}}^{\theta}\|u\|^{\alpha-\theta}_{L^{\bar{a}}_tL^{\bar{r}}_x} \|\nabla u \|_{L^{q}_tL^{r}_x},
\end{split}
\] 
where $
\frac{3}{d}-b-1=\frac{\theta(2-b)}{\alpha}-\frac{3}{r_1}
$. Again the last relation is the same relation as in \eqref{RI}. Therefore, applying the equivalence of Sobolev spaces (since $r<3$) together with \eqref{LG2ii1}, one has
$$\left\|\sqrt{\mathcal{L}_a} (|x|^{-b}|u|^\alpha u)\right\|_{L_t^{2}L^{\frac{6}{5}}_x}\lesssim   \|u\|_{L^\infty_t H_{a}^{1}}^{\theta}\|u\|^{\alpha-\theta}_{L^{\bar{a}}_tL^{\bar{r}}_x} \|\sqrt{\mathcal{L}_a} u \|_{L^{q}_tL^{r}_x}.
$$
Finally, we show (iii). Indeed, the interpolation inequality and (i)-(ii) imply that
\[
\begin{split}
\left\| \sqrt{\mathcal{L}_a}^{s_c}F(x,u)\right\|_{L^{2}_t L^{\frac{6}{5}}_x}&\lesssim \left\| F(x,u) \right\|^{1-s_c}_{L^{2}_t L^{\frac{6}{5}}_x}
\left\|\sqrt{\mathcal{L}_a} F(x,u)\right\|^{s_c}_{L^{2}_t L^{\frac{6}{5}}_x}\\
& \lesssim \| u\|^{\theta}_{L^\infty_tH^1_a}\|u\|^{\alpha-\theta}_{L^{\bar{a}}_tL^{\bar{r}}_x}\| u\|^{1-s_c}_{S(L^2)}\| \sqrt{\mathcal{L}_a} u\|^{s_c}_{S(L^2)}. 
\end{split}
\]
This completes the proof of the lemma.
\end{proof}
\end{lemma}
\begin{remark}\label{important*}
Note that, in Lemma \ref{LG2} the pair ($\bar{a},\bar{r})$ is not $S(\dot{H}^{s_c})$-admissible due to $\bar{r}<6$ not being true for $\alpha\geq 3-2b$. However, we obtain a small data global result assuming $\|e^{-it\mathcal{L}_a}u_0\|_{L^{\bar{a}}_tL^{\bar{r}}_x  }$ sufficiently  small. For the proof of this result (Theorem \ref{GWPN=3}) we do not use the Strichartz estimate  
\eqref{SE5} in view of $(\bar{a},\bar{r})$ not being $\dot{H}^{s_c}$-admissible. Instead, use the Sobolev embedding and apply the Strichartz estimate \eqref{SE3}. This is possible since the $S$-admissible pair used satisfies the conditions \eqref{CPA0}. It is also worth mentioning that, since we do not use \eqref{SE5} then the result holds for nonradial data. 
\end{remark}

\begin{proof}[\bf{Proof of Theorem \ref{GWPN=3}}]
We only show (i) since (ii) and (ii) are immediate consequences. As before, define 
$$
S=\{ u:\;\|u\|_{\tilde{S}(\dot{H}^{s_c})}\leq 2\|e^{-it\mathcal{L}_a}u_0\|_{\tilde{S}(\dot{H}^{s_c})}\quad\textnormal{and}\quad\|\langle \sqrt{\mathcal{L}_a}\rangle u\|_{S(L^2)}\leq 2c\|u_0\|_{H^1_a}\}.
$$
We prove that $G=G_{u_0}$ defined in \eqref{OPERATOR} is a contraction on $S$ equipped with the metric 
$$
d(u,v)=\|u-v\|_{S(L^2)}.
$$
Combining the Sobolev embedding, equivalence of Sobolev spaces and Strichartz estimate \eqref{SE3}, it follows that
\begin{eqnarray*}
\|G(u)\|_{\tilde{S}(\dot{H}^{s_c})}&\leq& \|e^{-it\mathcal{L}_a} u_0\|_{\tilde{S}(\dot{H}^{s_c})}+c\left\|D^{s_c}\int_0^t e^{-i(t-s)\mathcal{L}_a} F(x,u)ds\right\|_{L_t^{\bar{a}}L_x^{\bar{p}}}\\
&\leq&  \|e^{-it\mathcal{L}_a} u_0\|_{\tilde{S}(\dot{H}^{s_c})}+c\left\|\sqrt{\mathcal{L}_a}^{s_c} F(x,u)\right\|_{L^{2}_tL_x^{\frac{2N}{N+2}}},
\end{eqnarray*}
where\footnote{It is easy to see that $2<\bar{p}<\frac{3}{s_c}$ and since $\frac{4-2b}{3}<\alpha<4-2b$ we have that $(\bar{a},\bar{p})$ is $S$-admissible.} 
\begin{equation}\label{pbarra1}
 \bar{p}=\frac{6\alpha(\alpha-\theta)}{\alpha(3-2b-2\varepsilon)+2\alpha s_c(\alpha-\theta)-\theta(4-2
b)}.   
\end{equation} 
In addition,
\begin{equation}\label{GHs2}
\|G(u)\|_{S(L^2)}\leq c\|u_0\|_{L^2}+ c\| F(x,u) \|_{L^{2}_tL_x^{\frac{2N}{N+2}}}
\end{equation}
and	
\begin{equation}\label{GHs3}
\|\sqrt{\mathcal{L}_a} G(u)\|_{S(L^2)}\leq c \|\sqrt{\mathcal{L}_a} u_0\|_{L^2}+ c\|\sqrt{\mathcal{L}_a} F(x,u)\|_{L^{2}_tL_x^{\frac{2N}{N+2}}}.
\end{equation}
An application of Lemma \ref{LG2} together with the last inequalities yield, for any $u\in S$,
\begin{eqnarray*}
\|G(u)\|_{\tilde{S}(\dot{H}^{s_c})} & \leq & \|e^{-it\mathcal{L}_a} u_0\|_{\tilde{S}(\dot{H}^{s_c})}+c\|u\|^\theta_{L^\infty_tH^1_a}\| u \|^{\alpha-\theta}_{\widetilde{S}(\dot{H}^{s_c})} \|\langle \sqrt{\mathcal{L}_a}\rangle u\|_{S(L^2)}\\
& \leq & \|e^{-it\sqrt{\mathcal{L}_a}} u_0\|_{\tilde{S}(\dot{H}^{s_c})}+c^{\theta+2}2^{\alpha+1} M^{\theta+1}\|e^{-it\mathcal{L}_a}  u_0 \|^{\alpha-\theta}_{\widetilde{S}(\dot{H}^{s_c})}
\end{eqnarray*}
and
\begin{equation}
\begin{split}
\|\langle \sqrt{\mathcal{L}_a}\rangle G(u)\|_{S(L^2)} &\leq  c\|u_0\|_{H^1_a}+c\|u\|^\theta_{L^\infty_tH^1_a}\| u \|^{\alpha-\theta}_{\widetilde{S}(\dot{H}^{s_c})} \|\langle \sqrt{\mathcal{L}_a}\rangle u\|_{S(L^2)}\\
& \leq  c\|u_0\|_{H^1_a}+c^{\theta+2}2^{\alpha+1} M^{\theta}\|e^{-it\mathcal{L}_a}  u_0 \|^{\alpha-\theta}_{\widetilde{S}(\dot{H}^{s_c})}\|u_0\|_{H^1_a}.
\end{split}
\end{equation}

\ If $\| e^{it\mathcal{L}_a}u_0 \|_{\widetilde{S}(\dot{H}^{s_c})}<\delta$ with\footnote{Note that, to define $\delta$, we need the condition $\alpha>1$.} (since $\alpha>1$ and $\theta$ small) 
\begin{equation}\label{deltasd}
\delta\leq \min\left\{\sqrt[\alpha-1-\theta]{\frac{1}{2c^{\theta+2}2^{\alpha+1}M^{\theta+1}}}     , \sqrt[\alpha-\theta]{ \frac{1}{2c^{\theta+1}2^{\alpha+1}M^\theta}}\right\},
\end{equation}
implies that
$$\|G(u)\|_{\widetilde{S}(\dot{H}^{s_c})}\leq 2\| e^{-it\mathcal{L}_a}u_0 \|_{\widetilde{S}(\dot{H}^{s_c})}\quad \mbox{and}\quad \|\langle \sqrt{\mathcal{L}_a}\rangle G(u) \|_{S(L^2)}\leq 2c\|u_0\|_{H^1_a}.$$
therefore, $G(u)\in S$. 

\ To show that $G$ is a contraction on $S$, we repeat the above computations. Indeed, let $u,v\in S$ one has
\begin{equation}\label{C2GH1}
\|G(u)-G(v)\|_{S(L^2)}\leq 2^{\alpha+1}c^{\theta+1} M^\theta\|e^{-it\mathcal{L}_a}u_0 \|^{\alpha-\theta}_{B(\dot{H}^{s_c})} \|u-v\|_{S(L^2)}.
\end{equation}
By using the last inequalities and \eqref{deltasd} we obtain
$$
d(G(u),G(v)) \leq 2^{\alpha+1}c^{\theta+1} M^\theta\|e^{-it\mathcal{L}_a}u_0 \|^{\alpha-\theta}_{B(\dot{H}^{s_c})} \;d(u,v)\leq \frac{1}{2}d(u,v),
$$
which implies that  $G$ is also a contraction. Therefore, by the contraction mapping principle, $G$ has a unique fixed point $u\in S$.
\end{proof}

\ We end this section with the proof of Theorem \ref{GWP2}. To this end, we establish good estimates on the nonlinearity, for negative values of $a$.
\begin{lemma}\label{LG3} 
Let $0<b<\frac{6-N}{2}$. Assume that $(N,a,\alpha)$ satisfy \begin{equation}\label{LGa<0}
\begin{cases} a>-\frac{(N-2)^2}{4}\;\;\;\;\qquad \qquad \qquad \qquad\ \textnormal{if}\;\;N=3,\;\;\;\frac{4-2b}{3}<\alpha\leq\; 2-2b \quad \textnormal{and} \;\; 0\leq b<\frac{1}{2},\\
a>-\frac{(N-2)^2}{4}+\left(\frac{\alpha(N-2)-(2-2b)}{2(\alpha+1)}\right)^2\;\;\textnormal{if}\;\;3\leq N\leq 5,\;\;\; \max\{\frac{4-2b}{N},\frac{2-2b}{N-2},1\}<\alpha< \frac{4-2b}{N-2}.
\end{cases}
\end{equation}
Then the following statements hold
\begin{itemize}
\item [(i)] $
\left\| F(x,u,v)\right\|_{L^{2}_tL_x^{\frac{2N}{N+2}}}\;\lesssim\;
\| u\|^{\theta}_{L^\infty_tH^1_a}\|u\|^{\alpha-\theta}_{L^{\bar{a}}_tL^{\bar{r}}_x} \|v\|_{S(L^2)}$
\item [(ii)] $
\left\|\sqrt{\mathcal{L}_a} F(x,u)\right\|_{L^{2}_tL_x^{\frac{2N}{N+2}}}\;\lesssim\;
\| u\|^{\theta}_{L^\infty_tH^1_a}\|u\|^{\alpha-\theta}_{L^{\bar{a}}_tL^{\bar{r}}_x} \|\sqrt{\mathcal{L}_a}   u\|_{S(L^2)}$ 
\item [(iii)] $
\left\|\sqrt{\mathcal{L}_a}^{s_c} F(x,u)\right\|_{L^{2}_tL_x^{\frac{2N}{N+2}}} \;\lesssim\;
\| u\|^{\theta}_{L^\infty_tH^1_a}\|u\|^{\alpha-\theta}_{L^{\bar{a}}_tL^{\bar{r}}_x} \|\langle\sqrt{\mathcal{L}_a}  \rangle u\|_{S(L^2)}$, 
\end{itemize}
where $\theta$ is sufficiently small.
\begin{proof} In view of \eqref{LGa<0} we have two cases:\\
{\bf Case $1$.} For $N=3$, $\frac{4-2b}{N}<\alpha\leq 2-2b$ and $0\leq b<\frac{1}{2}$, we define
\begin{equation}
\bar{a}=\frac{2(\alpha-\theta)}{1-\theta}\;,\quad\; \bar{r}=\frac{3\alpha(\alpha-\theta)}{\alpha(1-b)-\theta(2-b-\alpha)}\;,\quad \;q=\frac{2}{\theta}\;\quad \textnormal{and}\;\quad r=\frac{6}{3-2\theta}.
\end{equation}
{\bf Case $2$.} For $3\leq N\leq 5$ and  $\max\{\frac{2-2b}{N-2}, \frac{4-2b}{N},1\}<\alpha< \frac{4-2b}{N-2}$, define
\begin{equation}
\bar{a}=\frac{4(\alpha+1)(\alpha-\theta)}{4-2b-\alpha(N-4)}\qquad \quad \bar{r}=\frac{2N\alpha(\alpha+1)(\alpha-\theta)}{\alpha^2(N-2b)-\theta(4-2b)(\alpha+1)}
\end{equation}
and
\begin{equation}
q=\frac{4(\alpha+1)}{\alpha(N-2)-2+2b}\quad \qquad r=\frac{2N(\alpha+1)}{2(\alpha+1)+N-2b}.
\end{equation}
Note that, the hypothesis \eqref{LGa<0} yields\footnote{Recalling, the equivalence of Sobolev spaces (Remark \ref{equivalence}) holds if $\frac{N}{N-\rho}<r<\frac{N}{1+\rho}$.} $\frac{N}{N-\rho}<r<\frac{N}{1+\rho}$ and $(q,r)$ is S-admissible. Moreover, since $\theta$ is small, $b<\frac{N}{2}$ and $\alpha>\frac{2-2b}{N-2}$, we get $2<r<N$.

\ The proof is similar as in Lemma \ref{LG2}. We first consider (i). Let $A=\{B,B^C\}$. It follows that (using $\frac{1}{2}=\frac{\alpha-\theta}{\bar{a}}+\frac{1}{q}$)
\[
\begin{split}
\left\|F(x,u,v)\right\|_{L^2_tL_x^{\frac{2N}{N+2}}(A)}&\leq  \||x|^{-b}\|_{L^\gamma(A)}\|u\|^\theta_{L^\infty_tL_x^{\theta r_1}}  \|u\|^{\alpha-\theta}_{L_t^{\bar{a}}L_x^{\bar{r}}}   \| v \|_{L^{q}_tL_x^{r}},
\end{split}
\]
where $\frac{N}{\gamma}=\frac{N+2}{2}-\frac{N}{r_1}-\frac{N(\alpha-\theta)}{\bar{r}}-\frac{N}{r}$. Hence, the values of $\bar{r}$ and $r$ imply
$$
\frac{N}{\gamma}-b=\frac{\theta(2-b)}{\alpha}-\frac{N}{r_1},
$$
which is the same relation as in \eqref{RI}, so arguing as in Lemma \ref{LG2} and since  ($q,r$) is $S$-admissible we deduce (i). Let us prove (ii). From the equivalence of Sobolev spaces we know that
\begin{equation}\label{derivada}
\left\|\sqrt{\mathcal{L}_a} (|x|^{-b}|u|^\alpha u)\right\|_{L_t^{2}L^{\frac{2N}{N+2}}_x}\leq \left\|\nabla (|x|^{-b}|u|^\alpha u)\right\|_{L_t^{2}L^{\frac{2N}{N+2}}_x}.    
\end{equation}
We only estimate $|x|^{b-1}|u|^\alpha u$. We deduce that (using $1=\frac{N}{r}-\frac{N}{r_3}$, with $r<N$)
\[
\begin{split}
\left\||x|^{-b-1}|u|^\alpha u\right\|_{L_t^{2}L^{\frac{2N}{N+2}}_x(A)} &\leq  \left\|\||x|^{-b-1}\|_{L^d(A)}\|u\|_{L_{x}^{\theta r_1}}^{\theta} \|u\|^{\alpha-\theta}_{L_{x}^{(\alpha-\theta)r_2}}\|u\|_{L^{r_3}_x}\right\|_{L^{2}_t}\\
& \lesssim   \|u\|_{L^\infty_t H_{x}^{1}}^{\theta}\|u\|^{\alpha-\theta}_{L^{\bar{a}}_tL^{\bar{r}}_x} \|\nabla u \|_{L^{q}_tL^{r}_x},
\end{split}
\] 
where we get the same relation as in \eqref{RI}, i.e., $
\frac{N}{d}-b-1=\frac{\theta(2-b)}{\alpha}-\frac{N}{r_1}
$. Therefore, again the equivalence of Sobolev spaces together with \eqref{derivada} lead to
$$\left\|\sqrt{\mathcal{L}_a} (|x|^{-b}|u|^\alpha u)\right\|_{L_t^{2}L^{\frac{2N}{N+2}}_x}\lesssim   \|u\|_{L^\infty_t H_{a}^{1}}^{\theta}\|u\|^{\alpha-\theta}_{L^{\bar{a}}_tL^{\bar{r}}_x} \|\sqrt{\mathcal{L}_a} u \|_{L^{q}_tL^{r}_x}.
$$
In addition, applying the interpolation inequality and (i) - (ii) we have (iii), that is
\[
\begin{split}
\left\| \sqrt{\mathcal{L}_a}^{s_c}F(x,u)\right\|_{L^{2}_t L^{\frac{2N}{N+2}}_x}
& \lesssim \| u\|^{\theta}_{L^\infty_tH^1_a}\|u\|^{\alpha-\theta}_{L^{\bar{a}}_tL^{\bar{r}}_x}\| u\|^{1-s_c}_{S(L^2)}\| \sqrt{\mathcal{L}_a} u\|^{s_c}_{S(L^2)}. 
\end{split}
\]
\end{proof}
\end{lemma}

\begin{proof}[\bf Proof of Theorem \ref{GWP2}]

First, we define 
\begin{equation}\label{pbarra2}
\bar{p}=\frac{N\alpha (\alpha-\theta)}{\alpha s_c(\alpha-\theta)+\alpha(1-b)-\theta(2-b-\alpha)},     
\end{equation}
if Case $1$ of Lemma \ref{LG3} holds and
\begin{equation}\label{pbarra3}
\bar{p}=\frac{2N\alpha(\alpha+1) (\alpha-\theta)}{2\alpha s_c(\alpha+1)(\alpha-\theta)+\alpha^2(N-2b)-\theta(4-2b)(\alpha+1)},    
\end{equation}
if Case $2$ of Lemma \ref{LG3} holds.
Note that,  One has that $(\bar{a},\bar{p})$ is $S$-admissible and $s_c=\frac{N}{\bar{p}}- \frac{N}{\bar{r}}$. Moreover, $\bar{p} <\frac{N}{s_c}$, for $\theta$ small enough and $2 < \bar{p} <\frac{2N}{N-2}$. Hence, applying Sobolev embedding we get
\begin{eqnarray*}
\|\int_0^t e^{-i(t-s)\mathcal{L}_a} F(x,u)ds\|_{L_t^{\bar{a}}L_x^{\bar{r}}}&\leq& \left\|D^{s_c}\int_0^t e^{-i(t-s)\mathcal{L}_a} F(x,u)ds\right\|_{L_t^{\bar{a}}L_x^{\bar{p}}}.
\end{eqnarray*}
where\footnote{Note that $\bar{a}\geq 2$ since $\alpha>1$, which is important to conclude that $(\bar{a},\bar{p})$ is $S$-admissible.} $\bar{a}$ and $\bar{r}$ are defined in Lemma \ref{LG3}. Thus, with the previous lemma in hand the rest the proof is exactly the same as in Theorem \ref{GWPN=3}.
\end{proof}

\section{Scattering and existence of wave operator}\label{Sec5}

Our goal here is to show Theorem \ref{Wave operator}. It gives us a criterion to establish scattering and the existence of wave operator. To this end, we use the estimates on the nonlinearity obtained in Section \ref{Sec4}. Before proving the theorem itself, we must point out that our estimates in Lemmas \ref{LG1}, \ref{LG2N=3a}, and \ref{LGa<0} also hold if we replace the norms (in time) on the whole $\mathbb{R}$ by a interval $I\subset \mathbb{R}$. To see this it is sufficient to observe that in all results the only estimates in time we used was the H\"older inequality. We will only prove the result by assuming the assumptions of Theorem \ref{GWP1} holds. The other cases are
dealt with similarly.

\begin{proof}[\bf{Proof of Theorem \ref{Wave operator}}] 

\ 

\ Let's first prove show (i). We claim that if $\|u\|_{S(\dot{H}^{s_c})}<+\infty$, then
\begin{equation}\label{SCATTER1}
\|\langle\sqrt{\mathcal{L}_a}\rangle u\|_{S(L^2)}<+\infty.
\end{equation}
Indeed, we only show $\|\langle\sqrt{\mathcal{L}_a}\rangle u\|_{S(L^2;[0,\infty))}<+\infty$. A similar analysis may be performed to see that $ \|\langle\sqrt{\mathcal{L}_a}\rangle u\|_{S(L^2;(\infty,0])}<+\infty$. Given $\delta>0$ we decompose the interval $[0,\infty)$ into $n$ intervals $I_j=[t_j,t_{j+1})$ such that $\|u\|_{S(\dot{H}^{s_c};I_j)}<\delta$, for all $j=1,\ldots,n$. On the time interval $I_j$ we consider the integral equation
\begin{equation}\label{SCATTER2}
u(t)=e^{-i(t-t_j)\mathcal{L}_a}u(t_j)+i\lambda\int_{t_j}^{t}e^{-i(t-s)\mathcal{L}_a}(|x|^{-b}|u|^\alpha u)(s)ds.
\end{equation}
Combining the Strichartz estimates \eqref{SE1} and \eqref{SE3} together with Lemma \ref{LG1}, it follows that
\begin{equation}\label{SCATTER3}
\begin{split}
\|\langle\sqrt{\mathcal{L}_a}\rangle u\|_{S(L^2;I_j)}&\leq c\|u(t_j)\|_{H^1_a}+ c\left\|\langle\sqrt{\mathcal{L}_a}\rangle(|x|^{-b}|u|^\alpha u) \right\|_{L^2_{I_j}L_x^{\frac{2N}{N+2}}}\\
& \leq c\|u(t_j)\|_{H^1_a}+c\| u\|^{\theta}_{L^\infty_{I_j}H^2_x}\|u\|^{\alpha-\theta}_{S(\dot{H}^{s_c};I_j)} \|  u\|_{S(L^2;I_j)}\\
&\leq c\|u(t_j)\|_{H^1_a} +cM^\theta \delta^{\alpha-\theta}\| \langle\sqrt{\mathcal{L}_a}\rangle u\|_{S(L^2;I_j)},
\end{split}
\end{equation}	 
where we have used the assumption $\sup_{t\in I_j}\|u(t)\|_{H^1_a}\leq M$. Taking  $\delta>0$ such that $ \eta^\theta\delta^{\alpha-\theta}<\frac{1}{2c}$ we obtain that $\|\langle\sqrt{\mathcal{L}_a}\rangle u\|_{S(L^2;I_j)} \leq 2cM,
$ and by summing over the n intervals, we conclude the proof of \eqref{SCATTER1}. 

Returning to the proof of the theorem, define
$$
\phi^+=u_0+i\lambda\int\limits_{0}^{+\infty}e^{i(s)\mathcal{L}_a}\left(|x|^{-b}|u|^\alpha u\right)(s)ds.
$$
Note that $\phi^+ \in H^1_a(\mathbb{R}^N)$. Following the above steps, one has
\begin{equation}
\|\langle \sqrt{\mathcal{L}_a}\rangle \phi^+\|_{L^2}\leq  c\| u_0\|_{H^1_a}+c\| u\|^{\theta}_{L^\infty_tH^1_a}\|u\|^{\alpha-\theta}_{S(\dot{H}^{s_c})} \| \langle \sqrt{\mathcal{L}_a}\rangle u\|_{S(L^2)}. 
\end{equation}
Hence, applying \eqref{SCATTER1} we get the desired result.
  
Since $u$ is a solution of \eqref{INLSa}, a simple inspection gives
$$
 u(t)-e^{-it\mathcal{L}_a}\phi^+=-i\lambda\int\limits_{t}^{+\infty}e^{-i(t-s)\mathcal{L}_a}|x|^{-b}(|u|^\alpha u)(s)ds,
$$
thus as before  
$$
 \|u(t)-e^{-it\mathcal{L}_a}\phi^+\|_{H^1_a}\leq c \| u \|^\theta_{L^\infty_tH^1_a}\| u \|^{\alpha-\theta}_{S(\dot{H}^{s_c};[t,\infty))}\|\langle \sqrt{\mathcal{L}_a}\rangle u\|_{S(L^2)}.
$$
By using \eqref{SCATTER1} we get  $\|u\|_{S(\dot{H}^{s_c};[t,\infty))}\rightarrow 0$ as $t \rightarrow +\infty$, this implies that 
\begin{equation}\label{scat1}
\|u(t)-e^{-it\mathcal{L}_a}\phi^+\|_{H^1_a}\rightarrow 0, \,\,\textnormal{as}\,\,t\rightarrow +\infty.
\end{equation}

\ In the same way, let $
\phi^-=u_0+i\lambda\int\limits_{0}^{-\infty}e^{i(s)\mathcal{L}_a}\left(|x|^{-b}|u|^\alpha u\right)(s)ds$, 
so that we obtain $\phi^{-} \in H^1_a$ and $\|u(t)-e^{-it\mathcal{L}_a}\phi^+\|_{H^1_a}\rightarrow 0, \,\,\textnormal{as}\,\,t\rightarrow -\infty$. This completes the proof of (i). 

\ 

\ We now consider (ii). We will divide the proof in two parts. We first look for a fixed point for the operator
\begin{equation}
G(w)(t)=- i\lambda \int_t^{+\infty} e^{-i(t-s)L_a}(|x|^{-b}|w+e^{-itL_a}\phi|^\alpha (w+e^{-itL_a}\phi)(s)ds,\;\;t\in I_T,
\end{equation}
where $I_T=[T,+\infty)$ for $T\gg 1$. In the sequel, we show that $u$ defined by
$$
u(t)=e^{-it\mathcal{L}_a}\phi+w(t)
$$
is a solution of \eqref{INLSa}. To this end, let us start by proving that $G$ has a fixed point in $S(T,\rho)$ given by
$$
S(T,\rho)=\{w\in C\left(I_T;H^1_a(\mathbb{R}^N)\right):\; \|w\|_{T}:=\|w\|_{S(\dot{H}^{s_c};I_T)}+\|\langle \sqrt{\mathcal{L}_a}\rangle w\|_{S(L^2;I_T)}\leq \rho   \}.
$$

\ The Strichartz estimates \eqref{SE3}, \eqref{SE5} and Lemmas \ref{LG1} yield  
\begin{align}\label{EWO1}    
\| G(w) \|_{S(\dot{H}^{s_c};I_T)}\;\lesssim\; &  \| w+e^{-it\mathcal{L}_a}\phi\|^\theta_{L^\infty_{T}H^1_a}\| w+e^{-it\mathcal{L}_a}\phi \|^{\alpha-\theta}_{S(\dot{H}^{s_c};I_T)}\| w+e^{-it\mathcal{L}_a}\phi \|_{S(\dot{H}^{s_c};I_T)}
\end{align}
and
\begin{align}\label{EWO3}
\|\langle \sqrt{\mathcal{L}_a}\rangle G(w) \|_{S(L^2;I_T)}\; \lesssim\; &  \| w+e^{-it\mathcal{L}_a}\phi\|^\theta_{L^\infty_{T}H^1_a}\| w+e^{-it\mathcal{L}_a}\phi \|^{\alpha-\theta}_{S(\dot{H}^{s_c};I_T)}\| \langle \sqrt{\mathcal{L}_a}\rangle (w+e^{-it\mathcal{L}_a}\phi) \|_{S(L^2;I_T)}.
\end{align}	  
Hence,
\begin{eqnarray*}
\|G(w)\|_{T} &\lesssim &  \| w+e^{-it\mathcal{L}_a}\phi\|^\theta_{L^\infty_{T}H^1_x}\| w+e^{-it\mathcal{L}_a}\phi \|^{\alpha-\theta}_{S(\dot{H}^{s_c};I_T)}\| w+e^{-it\mathcal{L}_a}\phi \|_{T} \nonumber.
\end{eqnarray*}
   
\ In view of\footnote{Note that \eqref{U(t)phi} might not be true in the norm $L^{\infty}_{I_T}L^{\frac{2N}{N-2s_c}}_x$ and for this reason we remove $\left(\infty,\frac{2N}{N-2s_c}\right)$ in the definition of $\dot{H}^{s_c}$-admissible pair. More precisely, as we use Lemma \ref{LG1} to the proof and we did not use this pair to prove it.} 
\begin{equation}\label{U(t)phi}
\| e^{-it\mathcal{L}_a}\phi \|_{S(\dot{H}^{s_c};I_T)}\rightarrow 0
\end{equation}
as $T\rightarrow +\infty$, we can find $T_0>0$ large enough and $\rho>0$ small enough such that $G$ is well defined on $S(T_0,\rho)$. In the same way, we show that $G$ is a contraction on $B(T_0,\rho)$. Therefore, there exists a unique $w\in S(T_0,\rho)$ such that $G(w)=w$. 
 
\ On the other hand, by using the fact that $
\| w+e^{-it\mathcal{L}_a}\phi\|_{L^\infty_{T}H^1_a}\leq \|w\|_{H^1} +\|\phi\|_{H^1}<+\infty
$
and \eqref{EWO1} we deduce
\begin{eqnarray*}    
\| w \|_{S(\dot{H}^{s_c};I_T)} &\lesssim &  \| w+e^{-it\mathcal{L}_a}\phi \|^{\alpha-\theta}_{S(\dot{H}^{s_c};I_T)}\| w+e^{-it\mathcal{L}_a}\phi \|_{S(\dot{H}^{s_c};I_T)}\\
 &\lesssim &  A\| w\|_{S(\dot{H}^{s_c};I_T)} +A\|e^{-it\mathcal{L}_a}\phi \|_{S(\dot{H}^{s_c};I_T)},
\end{eqnarray*}
where $A=\| w+e^{-it\mathcal{L}_a}\phi \|^{\alpha-\theta}_{S(\dot{H}^{s_c};I_T)}$. Observe that, if $\rho$ has been chosen small enough and since $\|e^{-it\mathcal{L}_a}\phi\|_{S(\dot{H}^{s_c};I_T)}$ is also sufficiently small for $T$ large, it follows that
$$
A\leq c\| w\|^{\alpha-\theta}_{S(\dot{H}^{s_c};I_T)}+c\| e^{-it\mathcal{L}_a}\phi \|^{\alpha-\theta}_{S(\dot{H}^{s_c};I_T)}<\frac{1}{2}.
$$
Thus, using the last two inequalities we obtain $
\frac{1}{2} \| w \|_{S(\dot{H}^{s_c};I_T)} \lesssim A \|e^{-it\mathcal{L}_a}\phi \|_{S(\dot{H}^{s_c};I_T)},
$ consequently
\begin{equation}\label{EWO5}
  \| w \|_{S(\dot{H}^{s_c};I_T)}\rightarrow 0\;\;\;\;\textnormal{as}\;\;\;\; T\rightarrow +\infty.
\end{equation} 
This also implies that \footnote{Here, we use the relations \eqref{EWO3} and $\eqref{EWO5}$.} $
\|\langle \sqrt{\mathcal{L}_a} \rangle
 w \|_{S(L^{2};I_T)}\rightarrow 0\;\;\;\;\textnormal{as}\;\;\;\; T\rightarrow +\infty,
 $
and so
\begin{equation}\label{EWO6}
\|w\|_{T}\rightarrow 0 \;\; \textnormal{as}\;\; T\rightarrow +\infty.  
\end{equation}
   
\ We now prove that $u(t)=e^{-it\mathcal{L}_a}\phi+w(t)$ satisfies \eqref{INLSa} in the time interval $[T_0,\infty)$. Indeed, we need to show that
 \begin{equation}\label{CWO} 
 u(t)=e^{-i(t-T_0)\mathcal{L}_a}u(T_0)+i\lambda\int_{T_0}^{t}e^{-i(t-s)\mathcal{L}_a}(|x|^{-b}|u|^\alpha u)sds,
 \end{equation}
 for all $t\in [T_0,\infty)$. To do that, in view of
 $$
 w(t)=- i\lambda \int_t^\infty e^{-i(t-s)\mathcal{L}_a}|x|^{-b}|w+e^{-it\mathcal{L}_a}\phi|^\alpha (w+e^{-it\mathcal{L}_a}\phi)(s)ds,
$$
we have
 \begin{eqnarray*}
e^{-i(T_0-t)\mathcal{L}_a} w(t)&=&- i \lambda\int_t^\infty e^{-i(T_0-s)\mathcal{L}_a}|x|^{-b}|w+e^{-it\mathcal{L}_a}\phi|^\alpha (w+e^{-it\mathcal{L}_a}\phi)(s)ds\\
 &=& i\lambda\int_{T_0}^t e^{-i(T_0-s)\mathcal{L}_a}|x|^{-b}|w+e^{-it\mathcal{L}_a}\phi|^\alpha (w+e^{-it\mathcal{L}_a}\phi)(s)ds+w(T_0),
 \end{eqnarray*}
 and so applying $e^{-i(t-T_0)\mathcal{L}_a}$ and adding $e^{-it\mathcal{L}_a}\phi$ on both sides, we obtain \eqref{CWO}. 
 Finally, adding $e^{-it\mathcal{L}_a}\phi$ to both sides of the last equation, we deduce \eqref{CWO}.
  
 \ Finally, since $u(t)=e^{-it\mathcal{L}_a}\phi+w$ and applying \eqref{EWO6}, we conclude that 
 \begin{align*}
 \|u(t)-e^{-it\mathcal{L}_a}\phi\|_{L^\infty_TH^1_x}=\|w\|_{L^\infty_TH^1_x}\leq c\|w\|_{S(L^2;I_T)}+c\|\nabla w\|_{S(L^2;I_T)}\leq c\|w\|_{T}\rightarrow 0\;\;\textnormal{as}\;\;T\rightarrow \infty,
\end{align*}
which completes the proof. 

\end{proof}

\bibliographystyle{abbrv}
\bibliography{bib}	

 \end{document}